\documentclass[12pt]{article}
\usepackage{graphicx,tipa}
\newcommand{\arc}[1]{{%
  \setbox9=\hbox{#1}%
  \ooalign{\resizebox{\wd9}{\height}{\texttoptiebar{\phantom{A}}}\cr#1}}}
\usepackage{color}
\usepackage{a4}
\usepackage{amsfonts}
\usepackage{amssymb}
\usepackage{xcolor}
\usepackage{soul}
\usepackage{amsmath,amsthm}
\usepackage{amsmath}
\usepackage{mathtools}
\usepackage{relsize}
\usepackage[symbol]{footmisc}
\providecommand{\U}[1]{\protect \rule{.1in}{.1in}}

\definecolor{ashgrey}{rgb}{0.7, 0.75, 0.81}
\newtheorem{theorem}{Theorem}
\newtheorem*{theorema}{Theorem A}

\newtheorem{claim}[theorem]{Claim}

\newtheorem{corollary}[theorem]{Corollary}

\newtheorem{observation}[theorem]{Observation}

\newtheorem{lemma}[theorem]{Lemma}

\newtheorem{proposition}[theorem]{Proposition}
\newtheorem{remark}[theorem]{Remark}

\begin{document}

\title{Gluing $CAT(0)$ domains}
\author{Charalampos Charitos, Ioannis Papadoperakis
\and and Georgios Tsapogas\thanks{Corresponding author}\\Agricultural University of Athens }
\maketitle

\begin{abstract}
In this work we describe a class of subsets of the Euclidean plane which, with the induced length metric, are locally $CAT(0)$ spaces and we show that the gluing of two such subsets along a piece of their boundary is again a locally $CAT(0)$ space provided that the sum of the signed curvatures at every gluing point is non-positive. A generalization to subsets of smooth Riemannian surfaces of curvature $k\leq 0$ is given.

\end{abstract}

\section{Introduction and Statements of Results\label{sec1}}

It is  standard in the theory of $CAT(k)$ spaces that a convex subset of a
$CAT(k)$ space is, with the induced length metric, again a $CAT(\kappa)$ space
and the gluing of two $CAT(k)$ spaces along a common convex subspace is again
$CAT(\kappa).$ We refer the reader to \cite{BrHa} for definitions and terminology concerning $CAT$ spaces.

In the context of Riemannian manifolds of dimension $n\geq 3,$ under certain assumptions on the sectional curvatures of the boundary, it is shown in \cite{Kov} that attaching Riemannian manifolds of curvature $\leq \kappa$ along some isometry of their boundaries results in a space of curvature $\leq \kappa .$

In this work  we address these matters in dimension $2$ without the convexity
assumption. More precisely, we describe a class of subspaces of $\mathbb{R}
^{2},$ not necessarily convex, which are locally $CAT(0)$ spaces and we impose
conditions under which the gluing of two such subsets along a piece of their
boundary is again a locally $CAT(0)$ space.

We begin with $\mathbb{R}^{2}$ being the ambient space and in Section 4 below we generalize all results in the case of gluing subsets of smooth Riemannian surfaces of curvature $k\leq 0 .$

Let $\Sigma$ be a $2-$dimensional connected complete sub-manifold of $\mathbb{R}^{2}$
whose boundary $\partial \Sigma$ consists of finitely many components each
being a piece-wise smooth curve in $\mathbb{R}^{2}.$ Equip $\Sigma$ with the induced from $\mathbb{R}^{2}$ length metric. The topology with respect to the induced length metric is equivalent to the relative topology from $\mathbb{R}^{2}. $ This follows from the fact that any two points in $\partial\Sigma$ have finite distance.  It follows that $\Sigma$ with the induced length metric is
complete and locally compact, hence, a geodesic metric space. We will be
calling such a space a \emph{domain} in $\mathbb{R}^{2}.~$

Let $X$ be an open simply connected subset of $\mathbb{R}^{2}$ equipped with the induced length metric. Let $\overline{X}$ be the enlargement of $X$ with all boundary points of finite distance from some, hence any, point in $X$. Then $\overline{X}$  is a $CAT(0)$ space. This is shown in a more general context in \cite{Bis, LyWe}. From this we immediately have the following

\begin{proposition}
\label{cat0subsets}A domain $\Sigma$ in $\mathbb{R}^{2}$ is a locally $CAT(0)$ space.
\end{proposition}

Let $\Sigma_{A},\Sigma_{B}$ be domains in $\mathbb{R}^{2}$ as described above.
Let $I_{A}$ (resp. $I_{B}$) be a closed (finite) subinterval of $\partial
\Sigma_{A}$ (resp. $\partial \Sigma_{B}$). Assume that $I_{A},I_{B}$ are
isometric and, as curves, are both parametrized by arc-length by the same real
interval $J$
\[
\sigma_{A}:J\rightarrow I_{A}\mathrm{\ and\ }\sigma_{B}:J\rightarrow I_{B}.
\]
Recall that the signed curvature of a curve which is the boundary of a domain is the curvature defined with respect to the 2-frame consisting of the tangent vector and the unit normal chosen to be directed towards the interior of the domain.

For each $s\in J$ denote by $\kappa_{A}\left(  s\right)  $ (resp. $\kappa
_{B}\left(  s\right)  $) the signed curvature of $I_{A}$ (resp. $I_{B}$) at
the point $\sigma_{A}\left(  s\right)  $ (resp. $\sigma_{B}\left(  s\right)
$). Assume the following properties hold for
all $s\in J:$

\begin{description}
\item[(k1)] $\kappa_{A}\left(  s\right)  \leq 0$ for all $s\in J.$

\item[(k2)] $\kappa_{B}\left(  s\right) \geq 0$ for all $s\in J.$

\item[(k3)] $\kappa_{A}\left(  s\right)  +\kappa_{B}\left(  s\right)  \leq 0$ where equality holds for finitely many points in $J.$
\end{description}

As $I_{A}$ and $I_{B}$ are isometric we may glue $\Sigma_{A}$ with $\Sigma_{B}$ along their isometric
boundary pieces $I_{A}\equiv I_{B}$ to form a connected surface $\Sigma$
\[
\Sigma:= \Sigma_{A}\cup_{I_{A}\equiv I_{B}}\Sigma_{B}
\]
with piece-wise smooth boundary $\partial \Sigma=\left(  \partial \Sigma
_{A}\backslash \mathrm{Int}\left(  I_{A}\right)  \right)  \cup \left(
\partial \Sigma_{B}\backslash \mathrm{Int}\left(  I_{B}\right)  \right)  $

As $\Sigma$ is a locally compact, complete length space, $\Sigma \ $is a
geodesic metric space. We will show in Section \ref{sec3} below the following

\begin{theorem}
\label{main}The surface $\Sigma$ is a locally $CAT(0)$ metric space.
\end{theorem}
\begin{remark}
 If instead of the assumption \textbf{\rmfamily{(k2)}} the curvature $\kappa_{B}$ satisfies $\kappa_{B}\left(  s\right) \leq 0$ then both $\Sigma_A , \Sigma_B$ with their induced length metric have the property $I_A$ and $I_B$ are geodesic boundaries of $\Sigma_A$ and $\Sigma_B$ respectively. Therefore, the conclusion of the above Theorem is a standard result.
\end{remark}
The necessity of assumption \textbf{(k3)} is exhibited by an example given at the end of Section \ref{sec3} below.

\section{Limits of $CAT(0)$ spaces\label{secLimits}}
In this section we will prove Proposition \ref{basicLemma} below which roughly says that the limit of $CAT(0)$ spaces is a $CAT(0)$ space. This Proposition will be used in the proof of Theorem \ref{main} in Section \ref{sec3} as well as in the proof of Theorem \ref{mainriemann} in Section \ref{sec4}. However, we state it separately as it can be interesting in its own right.


\begin{proposition}
\label{basicLemma}Let $ A, A_{k},k\in \mathbb{N}$ be locally compact geodesic metric spaces with the following properties

\begin{itemize}
\item[(a1)] Each $A_{k}$ is a $CAT(0)$ space.

\item[(a2)] For every $x\in A$ there exists $K=K(x)\in \mathbb{N}$ such
that $x\in A_{k}$ for all $k\geq K.$

\item[(a3)] For every $x,y\in A,$ $\lim_{k\rightarrow \infty}\left \vert x-y\right \vert _{k}=\left \vert x-y\right \vert .$
\end{itemize}
where $\left \vert \quad \right \vert $ denotes the metric in $A$ and $\left \vert
\quad \right \vert _{k}$ the metric in $A_{k}.$ \\
Then
$A$ is a $CAT(0)$ space.
\end{proposition}
\begin{proof}
We first show that A is uniquely geodesic. Assume, on the contrary, that there exist points $x,y\in A$ with two geodesic segments $\left[  x,y\right]_1,\left[x,y\right]_2$ joining them. Let $z_1$ (resp. $z_2$) be the midpoint of  $\left[  x,y\right]_1$ (resp.$\left[  x,y\right]_2$). We may assume that $z_1 \neq z_2$ (if equal, replace one of $x,y$ with $z_1$).

For all $k$ with $k\geq \max \left\{ K(x), K(y), K(z_1)\right\}$, denote by $\left[  x,y\right] _k $ the unique geodesic segment in $A_k$ with endpoints $x,y.$ Let $m_k$ be the midpoint of $\left[  x,y\right] _k .$ Then by (a3) we have\\
\begin{equation}
\begin{array}{l}
\lim_{k\rightarrow \infty}\left \vert x-z_1\right \vert _{k}
=
\lim_{k\rightarrow \infty}\left \vert y-z_1\right \vert _{k}
= \displaystyle
\frac{\left \vert x-y\right \vert}{2}  \\
\lim_{k\rightarrow \infty}\left \vert x-m_k\right \vert _{k}
=\displaystyle
\lim_{k\rightarrow \infty}\left \vert y-m_k\right \vert _{k}
=\lim_{k\rightarrow \infty}\frac{\left \vert x-y\right \vert _k}{2}
=\frac{\left \vert x-y\right \vert}{2}.
\end{array}
\label{limits4}
\end{equation}
For the geodesic triangle $T_k\left(  x,z_1,y\right)  $  in $A_k$ with vertices $x,z_1 ,y$ consider the corresponding comparison triangle  $\overline{T_k}$  in
$\mathbb{R}^{2}$ with vertices $\overline{x}_k,\overline{z_1}_k,\overline{y}_k.$ By (\ref{limits4}) the sides of $\overline{T_k}$ satisfy
\[
\lim_{k\rightarrow \infty}\left \Vert \overline{x}_k - \overline{y}_k\right \Vert = \left \vert x-y\right \vert
\mathrm{\ \ and\ \ }
\lim_{k\rightarrow \infty}\left \Vert \overline{x}_k
- \overline{z_1}_k \right \Vert
=
\lim_{k\rightarrow \infty}\left \Vert
\overline{y}_k - \overline{z_1}_k \right \Vert
= \frac{\left \vert x-y\right \vert}{2}.
\]
where $\left \Vert \cdot \right \Vert $ denotes the Euclidean distance.\\
It follows that $\left \Vert \overline{m}_k - \overline{z_1}_k \right \Vert \rightarrow 0$ and, thus,
$\left \vert z_1 -m_k\right \vert \rightarrow 0. $ Similarly we show $\left \vert z_2 -m_k\right \vert \rightarrow 0$ which implies that $\left \vert z_1 -z_2\right \vert \rightarrow 0$ a contradiction because $z_1 \neq z_2 .$
\begin{center}
\begin{figure}[h]
\hspace*{27mm}\includegraphics[scale=0.73]{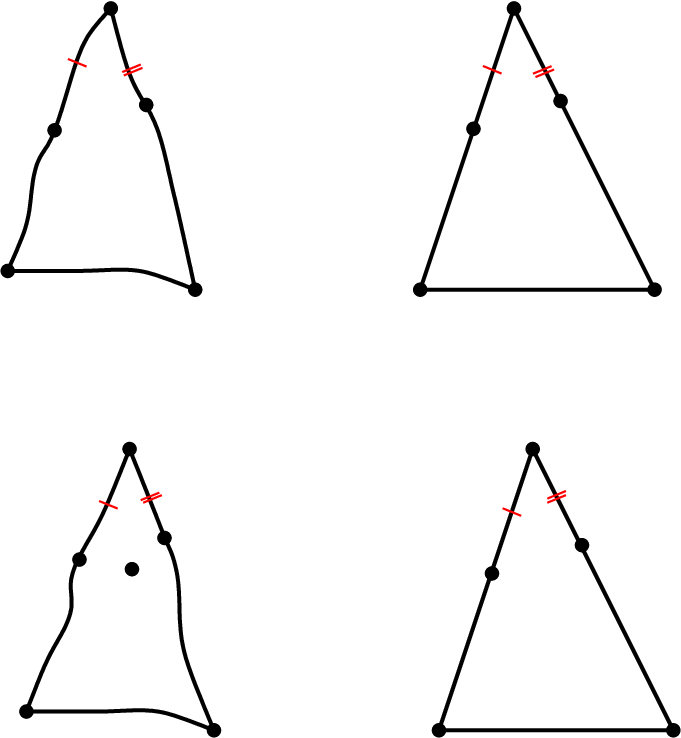}
\newline \begin{picture}(22,12)
\put(110,279){$x$}     \put(85,232){$a$}
\put(70,175){$y$}       \put(133,239){$b$}
\put(148,168){$z$}
\put(250,279){$\overline{x}$}     \put(232,232){$\overline{a}$}
\put(213,172){$\overline{y}$}     \put(280,237){$\overline{b}$}
\put(311,172){$\overline{z}$}
\put(144,195){$T\subset A$} \put(300,195){$\overline{T}\subset \mathbb{R}^2$}
\put(117,125){$x$}     \put(90,82){$a_k$}
\put(75,25){$y$}       \put(140,87){$b_k$}
\put(155,16){$z$}       \put(119,68){$a$}
\put(254,125){$\overline{x_k}$}     \put(234,78){$\overline{a_k}$}
\put(215,20){$\overline{y_k}$}     \put(286,86){$\overline{b_k}$}
\put(318,18){$\overline{z_k}$}
\put(145,45){$T_k \subset A_k$}\put(305,43){$\overline{T_k}
\subset \mathbb{R}^2$}
\end{picture}
\caption{The comparison triangles $\overline{T}$ and $\overline{T_{k}}$ for
the geodesic triangles $T\subset A$ and $T_{k} \subset A_{k} $ respectively.}%
\label{comparison1}%
\end{figure}
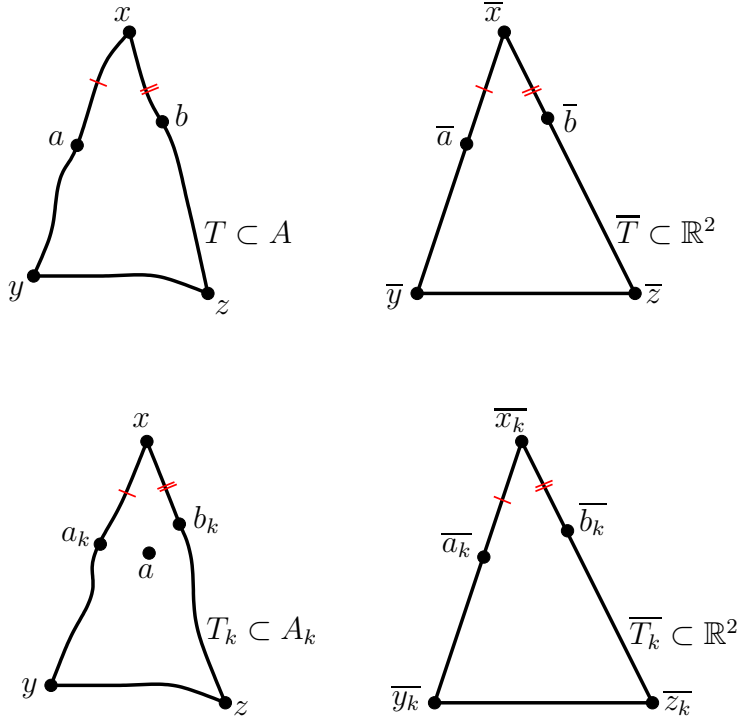
\end{center}
We proceed to show
that $CAT(0)$ inequality holds for any three points $x,y,z\in A.$ Denote by
$T\left(  x,y,z\right)  $ the geodesic triangle in $A$ with vertices $x,y,z.$
As usual, $\overline{T}$ denotes the corresponding comparison triangle in
$\mathbb{R}^{2}$ with vertices $\overline{x},\overline{y},\overline{z}.$ For
arbitrary points $a\in \left[  x,y\right]  ,b\in \left[  x,z\right]  $ we will
show
\begin{equation}
\left \vert a-b\right \vert \leq \left \Vert \overline{a}-\overline{b}\right \Vert .
\label{ccat}%
\end{equation}
Let $T_{k},$ for all $k$ sufficiently large, be the geodesic triangle in
$A_{k}$ with vertices $x,y,z,$ that is, $T_{k}=\left[  x,y\right]  _{k}%
\cup \left[  y,z\right]  _{k}\cup \left[  z,x\right]  _{k}.$ Set $a_{k}$ to be
the point in $\left[  x,y\right]  _{k}$ with $\left \vert x-a_{k}\right \vert
_{k}=\left \vert x-a\right \vert $ and $b_{k}$ the point in $\left[  x,z\right]
_{k}$ with $\left \vert x-b_{k}\right \vert _{k}=\left \vert x-b\right \vert .$
The corresponding comparison triangle $\overline{T_{k}}$ in $\mathbb{R}^{2}$
for $T_{k}$ has vertices $\overline{x_{k}},\overline{y_{k}},\overline{z_{k}}$
and the points corresponding to $a_{k},b_{k}$ are $\overline{a_{k}}\in \left[
\overline{x_{k}},\overline{y_{k}}\right]  $ and $\overline{b_{k}}\in \left[
\overline{x_{k}},\overline{z_{k}}\right]  .$ All the above notation is given, for convenience, in Figure \ref{comparison1}. We will need the following three Claims

\underline{CLAIM 1}: $\left \Vert \overline{a_{k}}-\overline{b_{k}}\right \Vert
\rightarrow \left \Vert \overline{a}-\overline{b}\right \Vert $ as $k\rightarrow
\infty.$

\underline{CLAIM 2}: $\left \vert a_{k}-a\right \vert _{k}\rightarrow0$ and
$\left \vert b_{k}-b\right \vert _{k}\rightarrow0$ as $k\rightarrow \infty.$

\underline{CLAIM 3}: $\left \vert a_{k}-b_{k}\right \vert _{k}\rightarrow
\left \vert a-b\right \vert $ as $k\rightarrow \infty.$

Assuming the above Claims the proof of (\ref{ccat}) follows easily: assume, on
the contrary, that $\left \vert a-b\right \vert >\left \Vert \overline
{a}-\overline{b}\right \Vert .$ Then by Claims 1 and 3 and for large enough
$k_{0},$ we have $\left \vert a_{k_{0}}-b_{k_{0}}\right \vert _{k_{0}}>$
$\left \Vert \overline{a_{k_{0}}}-\overline{b_{k_{0}}}\right \Vert $ which
contradicts the fact that $A_{k_{0}}$ is a $CAT(0)$ space.
\begin{center}
\begin{figure}[ptb]
\hspace*{27mm}\includegraphics[scale=0.93]{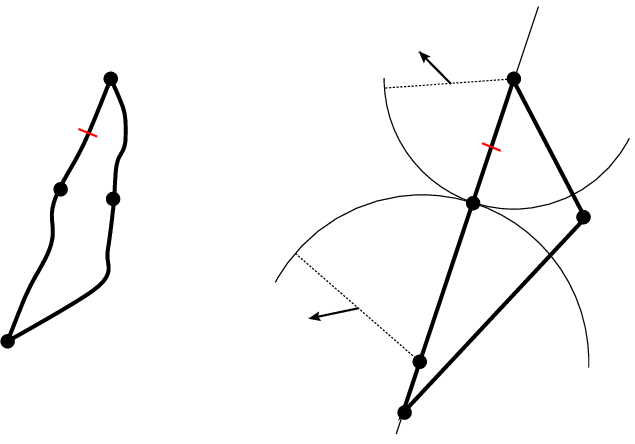}
\newline \begin{picture}(22,12)
\put(120,188){$x$} \put(237,197){radius $\left \vert x-a\right \vert$}
\put(133,128){$a$}     \put(88,135){$a_k$}
\put(71,55){$y$}
\put(263,25){$\left(  \overline{y}\right)  _{k}$}
\put(343,116){$\left(  \overline{a}\right)  _{k}$}
\put(250,52){$y_0$}
\put(312,180){$\left(  \overline{x}\right)  _{k}$}
\put(320,216){$\Lambda$}
\put(165,62){radius $\left \vert y-a\right \vert$}
\put(278,133){$\overline{a_k}$}
\end{picture}
\caption{The geodesic triangle $\bigtriangleup_{k}$ in $A_{k}$ with vertices
$x,y,a$ and its comparison triangle $\overline{\bigtriangleup}_{k}$.}%
\label{comparison2}%
\end{figure}
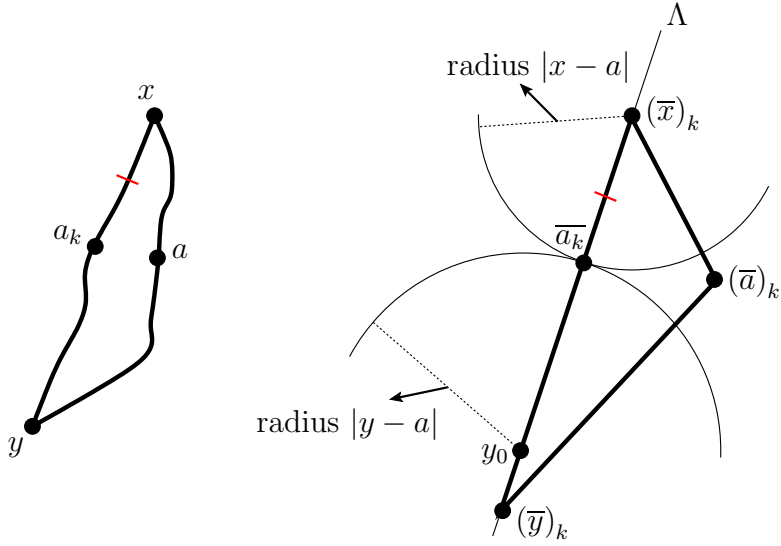
\end{center}
The proof of Claim 1 is straightforward: by property (a3), $\left \vert
x-y\right \vert _{k}\rightarrow \left \vert x-y\right \vert $ and, thus,
\[
\left \Vert \overline{x_{k}}-\overline{y_{k}}\right \Vert \rightarrow \left \Vert
\overline{x}-\overline{y}\right \Vert
\]
and similarly,
\[
\left \Vert \overline{x_{k}}-\overline{z_{k}}\right \Vert \rightarrow \left \Vert
\overline{x}-\overline{z}\right \Vert \mathrm{\ and\ }\left \Vert \overline
{z_{k}}-\overline{y_{k}}\right \Vert \rightarrow \left \Vert \overline
{z}-\overline{y}\right \Vert .
\]
In other words, the lengths of the sides of the triangles $\overline{T_{k}}$
converge to the corresponding lengths of the sides of the triangle
$\overline{T}.$ Since $a_{k},b_{k}$ were chosen so that%
\[
\left.
\begin{array}
[c]{c}%
\left \vert x-a_{k}\right \vert _{k}=\left \vert x-a\right \vert \\
\mathrm{and}\\
\left \vert x-b_{k}\right \vert _{k}=\left \vert x-b\right \vert
\end{array}
\right \}  \Longrightarrow%
\begin{array}
[c]{c}%
\left \Vert \overline{x_{k}}-\overline{a_{k}}\right \Vert =\left \Vert
\overline{x}-\overline{a}\right \Vert \\
\mathrm{and}\\
\left \Vert \overline{x_{k}}-\overline{b_{k}}\right \Vert =\left \Vert
\overline{x}-\overline{b}\right \Vert
\end{array}
\]
and a Euclidean triangle is determined by the lengths of its sides, it
follows that $\left \Vert \overline{a_{k}}-\overline{b_{k}}\right \Vert
\rightarrow \left \Vert \overline{a}-\overline{b}\right \Vert .$

The proof of Claim 3 follows immediately from property (a3), Claim 2 and the
following inequality
\[
\Bigl \vert \left \vert a_{k}-b_{k}\right \vert _{k}-\left \vert a-b\right \vert
_{k}\Bigr \vert \leq \left \vert a_{k}-a\right \vert _{k}+\left \vert
b_{k}-b\right \vert _{k}\rightarrow0.
\]
We conclude the proof of the Proposition by showing Claim 2. For each $k$ large
enough, consider the geodesic triangle $\bigtriangleup_{k}$ in $A_{k}$ with
vertices $x,y,a$ and its comparison triangle $\overline{\bigtriangleup}_{k}$
in $\mathbb{R}^{2}$ with vertices $\left(  \overline{x}\right)  _{k},\left(
\overline{y}\right)  _{k},\left(  \overline{a}\right)  _{k}.$ The point
$a_{k}\in$ $\left[  x,y\right]  _{k}$ was chosen so that $\left \vert
x-a_{k}\right \vert _{k}=\left \vert x-a\right \vert $ hence $\left \Vert \left(
\overline{x}\right)  _{k}-\overline{a_{k}}\right \Vert $ is constant for all
$k.$ In particular, we may choose all comparison triangles $\overline
{\bigtriangleup}_{k}\subset \mathbb{R}^{2}$ to have the segment $\left[
\left(  \overline{x}\right)  _{k},\overline{a_{k}}\right]  $ in common.
Moreover, the vertices $\left(  \overline{y}\right)  _{k}$ belong to the
(Euclidean) line $\Lambda$ containing $\left[  \left(  \overline{x}\right)
_{k},\overline{a_{k}}\right]  $ for all $k, $ see Figure \ref{comparison2}.

By $CAT(0)$ inequality $\left \vert a_{k}-a\right \vert _{k}\leq \left \Vert
\overline{a_{k}}-\left(  \overline{a}\right)  _{k}\right \Vert $ so it suffices
to show that $\left \Vert \overline{a_{k}}-\left(  \overline{a}\right)
_{k}\right \Vert \rightarrow0.$

Since, by property (a3), $\left \vert x-a\right \vert _{k}\rightarrow \left \vert
x-a\right \vert ,$ it follows
\begin{equation}
\left \Vert \left(  \overline{x}\right)  _{k}-\left(  \overline{a}\right)
_{k}\right \Vert \rightarrow \left \vert x-a\right \vert . \label{ccat1}%
\end{equation}
Moreover,%
\begin{equation}
\left \vert y-a_{k}\right \vert _{k}\rightarrow \left \vert y-a\right \vert
\label{ccat4}%
\end{equation}
because
\[%
\begin{split}
\left \vert y-a_{k}\right \vert _{k}=\left \vert x-y\right \vert _{k}- \left \vert
x-a_{k}\right \vert _{k} = \left \vert x-y\right \vert _{k}-\left \vert
x-a\right \vert \xrightarrow{\text{\ Claim\ 1\ } } \left \vert x-y\right \vert -
\left \vert x-a\right \vert =\\
=\left \vert y-a\right \vert .
\end{split}
\]
Now (\ref{ccat4}) implies that
\begin{equation}
\left \Vert \left(  \overline{y}\right)  _{k}-\overline{a_{k}}\right \Vert
\rightarrow \left \vert y-a\right \vert . \label{ccat2}%
\end{equation}
As $\left \{  \left(  \overline{y}\right)  _{k}\right \}  \subset \Lambda$ there
must exist a unique $y_{0}\in \Lambda$ such that
\begin{equation}
\left \Vert \left(  \overline{y}\right)  _{k}-y_{0}\right \Vert \rightarrow
0\mathrm{\ and\ }\left \Vert y_{0}-\overline{a_{k}}\right \Vert =\left \vert
y-a\right \vert . \label{ccat3}%
\end{equation}
Again by Claim 1,
\begin{equation}
\left \Vert \left(  \overline{y}\right)  _{k}-\left(  \overline{a}\right)
_{k}\right \Vert =\left \vert y-a\right \vert _{k}\rightarrow \left \vert
y-a\right \vert . \label{ccat5}%
\end{equation}
By (\ref{ccat3}) and (\ref{ccat5}), we have
\begin{equation}
\left \Vert y_{0}-\left(  \overline{a}\right)  _{k}\right \Vert \rightarrow
\left \vert y-a\right \vert . \label{ccat6}%
\end{equation}
Properties (\ref{ccat1}) and (\ref{ccat6}) assert that the points $\left \{
\left(  \overline{a}\right)  _{k}\right \}  $ must accumulate on the circle
centered at the (constant for all $k$) point $\left(  \overline{x}\right)
_{k}$ with radius $\left \vert x-a\right \vert $ as well as on the circle
centered at $y_{0}$ with radius $\left \vert y-a\right \vert .$ As the
(Euclidean) segment $\left[  \left(  \overline{x}\right)  _{k},y_{0}\right]  $
has length
\[
\left \Vert \left(  \overline{x}\right)  _{k}-\overline{a_{k}}\right \Vert
+\left \Vert \overline{a_{k}}-y_{0}\right \Vert =\left \vert x-a\right \vert
+\left \vert a-y\right \vert
\]
it follows that $\left \Vert \overline{a_{k}}-\left(  \overline{a}\right)
_{k}\right \Vert \rightarrow0$ as desired. This completes the proof of the
Proposition.
\end{proof}
\begin{remark}
 The above Proposition holds true with identical proof in the case $A_k, k\in\mathbb{N}$ are locally compact $CAT(\kappa)$ spaces, for any $\kappa \leq 0,$ with the conclusion being that $A$ is a $CAT(\kappa)$ space.
\end{remark}
\section{Gluing Domains from $\mathbb{R}^{2}$\label{sec3}}

Let $\Sigma_{A},\Sigma_{B}$ be domains in $\mathbb{R}^{2}$ satisfying
properties (k1), (k2) and (k3) and let $\Sigma$ be the surface
\[
\Sigma:=\Sigma_{A} \cup_{I_{A}\equiv I_{B}}  \Sigma_{B} .
\]
as described in the Introduction. We denote by $I_{\Sigma}$ the curve in
$\Sigma$ corresponding to $I_{A}\equiv I_{B}.$ Note that $I_{\Sigma}$ is
properly embedded in $\Sigma,$ that is, $\partial I_{\Sigma}$ consists of two
points in the boundary of $\Sigma$ and $I_{\Sigma}\cap \partial \Sigma=\partial
I_{\Sigma}.$ Clearly $I_{\Sigma}$ is again parametrized by $J$ and we will denote this arc-length parametrization by
\[\sigma :J\rightarrow I_{\Sigma} \] so that for every $s\in J$ the above gluing gives $\sigma_A (s)\equiv \sigma_B (s) \equiv \sigma (s).$

Before the proof of Theorem \ref{main}, we show the following

\begin{proposition}
\label{ugl}Geodesic segments in $\Sigma$ are locally unique.
\end{proposition}
In the proof of the above Proposition we will need the following property for the geodesic metric space $\Sigma .$
\begin{lemma}\label{piangle}
 Let $x\in \Sigma_A \setminus I_{A} , y \in \Sigma_B \setminus I_{B}  $ and $z_A\in I_{A}, z_B\in I_B$ be points such that
 \begin{itemize}
  \item $z_A$ is identified  with  $z_B$ under the identification $I_{A}\equiv I_B$ to give a point in $I_{\Sigma}\subset \Sigma$ denoted by $z$ and
  \item the union of the geodesic segments $[x,z]$ and $[z,y]$ is a geodesic at $z$ for the metric space $\Sigma .$
 \end{itemize}
 Then the angles $\theta_A$ (resp. $\theta_B$) formed by $[z_A,x]$ (resp. $[z_B,y]$) and the tangent of $\partial \Sigma_A$ (resp. $\partial \Sigma_B$) at $z_A$ (resp. $z_B$) satisfy\[ \theta_A +\theta_B = \pi. \]
\end{lemma}
For the proof of the above Lemma we will use the following straightforward facts:
\begin{observation}\label{obs1}
Let $\theta_A , \theta_B$ be non-zero angles satisfying $\theta_A +\theta_B < \pi.$ Consider the Euclidean isosceles triangle $\triangle(O,A,B)$ with $ \left\vert OA \right\vert=\left\vert OB \right\vert=1$ and $A\widehat{O}B =\theta_A + \theta_B. $ The unique point $E\in AB$ such that $A\widehat{O}E=\theta_A$ and $E\widehat{O}B=\theta_B$ has distance $\rho=\left\vert OE \right\vert$ from $O$ with $\rho$  uniquely determined by $\theta_A , \theta_B .$
\end{observation}
\begin{observation}\label{obs2}
Let $\sigma_1 , \sigma_2$ be two smooth curves in $\mathbb{R}^2$ defined on an interval containing $0\in\mathbb{R},$ such that  $\sigma_1 (0)= \sigma_2(0)$ and the tangent vectors $\sigma_1^{\prime} (0), \sigma_2^{\prime}(0)$ are linearly independent. Then there exists $s_0>0$ such that for all $s<s_0$ the images $\mathrm{Im}\,\sigma_1|_{[0,s]}$, $\mathrm{Im}\,\sigma_2|_{[0,s]}$ and the geodesic segment $\left[\sigma_1(s), \sigma_2(s)\right]$ have disjoint interiors.
\end{observation}
\begin{proof}[Proof of Lemma \ref{piangle}]
Assume, on the contrary, that $\theta_A +\theta_B < \pi. $

Let $\alpha : \left[ 0,\left\vert z_A-x \right\vert \right]\rightarrow \Sigma_A$ be an arc length parametrization of the geodesic segment $[z_A,x]$ and, similarly, $\beta : \left[ 0,\left\vert z_B-y \right\vert \right]\rightarrow \Sigma_B$ for the geodesic segment $[z_B,y].$ As $[z_A,x]$ cannot be tangent to $\partial \Sigma_A ,$ observation \ref{obs2} applies to the curves $\alpha$ and $\sigma_A$ and provides a positive $s_A.$ Similarly, let $s_B$ be the positive number provided by Observation \ref{obs2} applied to the curves $\beta$ and $\sigma_B$.

For all $s<\min\left\{s_A,s_B\right\}$ let $\sigma_A(\rho s), \sigma_B(\rho s)$ be the points on $\partial \Sigma_A ,\partial \Sigma_B $ respectively, determined by the number $\rho$ provided by Observation \ref{obs1} which can be used because we assumed that that $\theta_A +\theta_B < \pi. $ Note that the points $\sigma_A(\rho s)$ and $ \sigma_B(\rho s), $ $s<\min\left\{s_A,s_B\right\}$ coincide under the identification $I_A \equiv I_B. $ By Observation \ref{obs2}, the union of the segments
\[ \gamma_s \equiv \left[ \alpha(s), \sigma_A(\rho s)\right]
\cup
\left[ \beta(s), \sigma_B(\rho s)\right]
\]
is a curve in $\Sigma$ which intersects $I_{\Sigma}$ transversely at the point $\sigma_A(\rho s)\equiv \sigma_B(\rho s).$ We will show that for small enough $s,$ the length of the curve $\gamma_s$ is smaller than
the length of $[ \alpha(s),z]\cup [z,\beta(s)]$ contradicting the assumption that the union of the geodesic segments $[x,z]$ and $[z,y]$ is a geodesic.

We form the following ratio and examine its limit as $s\rightarrow 0:$
\begin{align}\frac{\textrm{Length}\,(\gamma_s)}{\textrm{Length}\,([ \alpha(s),z_A]\cup [z_B,\beta(s)])} &=
\frac{\left\vert \alpha (s) - \sigma_A(\rho s)\right\vert +
          \left\vert \beta (s) - \sigma_B(\rho s)\right\vert
}
{\left\vert \alpha (s) -z_A \right\vert + \left\vert \beta (s) -z_B \right\vert} \nonumber
\\[4mm]
  & = \frac{1}{2}\left\vert \frac{\alpha (s)}{s} - \frac{\sigma_A(\rho s)}{s}\right\vert +
  \frac{1}{2}\left\vert \frac{\beta (s)}{s} - \frac{\sigma_B(\rho s)}{s}\right\vert
  \nonumber
\end{align}
where we used that $\alpha, \beta$ are parametrized by arc-length. hence,
\[\left\vert \alpha (s) -z_A \right\vert=s=\left\vert \beta (s) -z_B \right\vert.\]
The limit of the right hand side as $s\rightarrow 0$ is equal to
\begin{align}
 \frac{1}{2}\left\vert \alpha^{\prime} (0) -\sigma_{A}^{\prime} (0) \rho \right\vert +
 \frac{1}{2}\left\vert \beta^{\prime} (0) -\sigma_{B}^{\prime} (0) \rho \right\vert  & =  \frac{1}{2}\left\vert AE \right\vert +
 \frac{1}{2}\left\vert EB \right\vert \nonumber\\
 & <
 \frac{1}{2} \left(\left\vert OA \right\vert +\left\vert OB \right\vert \right)= \frac{1}{2} \left(1+1\right) =1\nonumber
\end{align}
where $OA,OB$ and $AEB$ are the sides of the triangle defined in Observation \ref{obs1}.
\end{proof}
\begin{proof}[Proof of Proposition \ref{ugl}]
We only need to check uniqueness of geodesic segments for a sufficiently small simply connected ball $N$ centered
at a point in $I_{\Sigma}.$ We first show the following \\
\underline{CLAIM}: Let $z,w$ be two distinct points in $I_{\Sigma}\cap N, $ and consider points $x\in \left(\Sigma_{A}\cap N\right)\setminus I_{\Sigma}$ and $y_z ,y_w\in \left(\Sigma_{B}\cap N\right)\setminus I_{\Sigma}.$ Assume that
\begin{align}
 &\lbrack x,y_z\rbrack =
\lbrack x,z\rbrack\cup\lbrack z,y_z\rbrack \nonumber \\
\textrm{and\ } &
\lbrack x,y_w\rbrack =\lbrack x,w\rbrack\cup\lbrack w,y_w\rbrack
\nonumber
\end{align}
are both geodesic segments in $\Sigma$ with $\lbrack z,y_z\rbrack\cap \lbrack w,y_w\rbrack = \emptyset. $ Then
\[ w\widehat{z}y_z +z\widehat{w}y_w > \pi.\]
\begin{center}
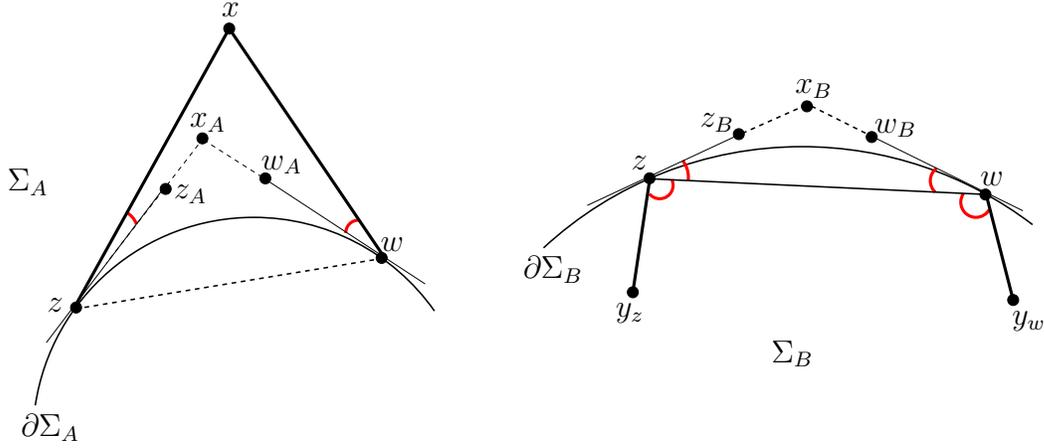
\begin{figure}[ptb]
\hspace*{7mm}\includegraphics[scale=0.65]{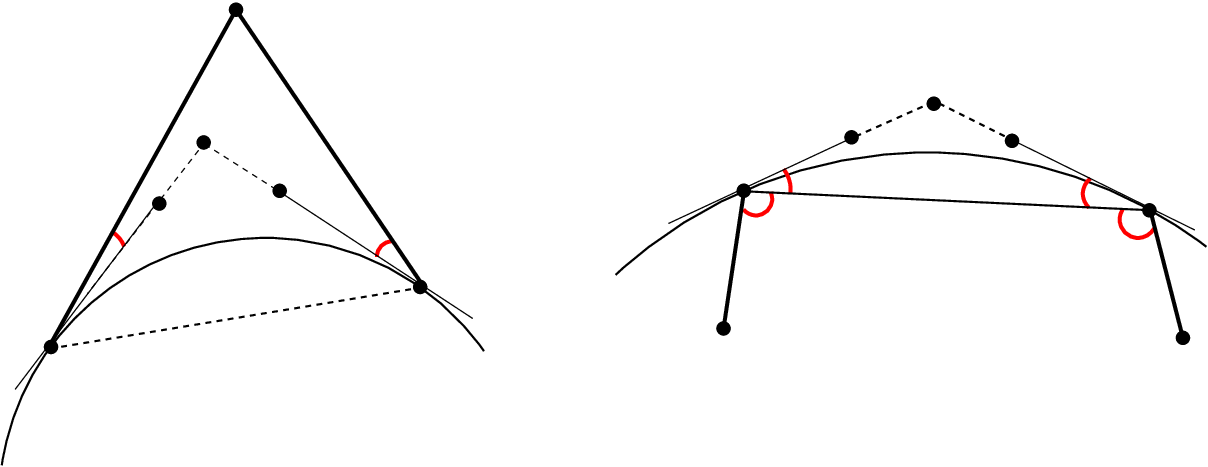}
\newline \begin{picture}(22,12)
\put(91,166){$x$}
\put(308,137){$x_B$}
\put(272,124){$z_B$}
\put(338,122){$w_B$}
\put(240,52){$y_z$}
\put(390,49){$y_w$}
\put(79,124){$x_A$}
\put(73,97){$z_A$}
\put(106,108){$w_A$}
\put(25,54){$z$}\put(246,108){$z$}
\put(151,77){$w$}\put(378,102){$w$}
\put(15,7){$ \partial \Sigma_A$}
\put(10,100){$ \Sigma_A$}\put(299,35){$ \Sigma_B$}
\put(205,67){$ \partial \Sigma_B$}
\end{picture}
\caption{The tangent lines at $z,w$ and the angles involved in the proof of the Claim in Proposition \ref{ugl}.}
\label{n_gon}
\end{figure}
\end{center}
\underline{Proof of Claim}:  Denote by $I_{zw}$ the
subinterval of $I_{\Sigma}\cap N$ with endpoints $z,w$ which is parametrized by a subinterval of $\subset J$ denoted by $J_{zw}.$ By property (k3) we have
\begin{equation}
{\displaystyle \int \limits_{J_{zw}}}\left \vert \kappa_{A}\left(  s\right)
\right \vert ds>{\displaystyle \int \limits_{J_{zw}}}\kappa_{B}\left(  s\right)
ds\label{UG0}%
\end{equation}
The lines tangent to $I_{A}$ at the points $z,w$ intersect at a point in $\Sigma_{A}$, say $x_{A},$ contained in the (Euclidean) triangle with vertices $z,x,w$ (see Figure \ref{n_gon}). In view of the generalization to Riemann surfaces in Section \ref{sec4} below, we will not use in our proof the intersection point $x_A.$ Instead, we pick points
$z_A \in \lbrack z,x_A\rbrack$ and $w_A \in \lbrack w,x_A\rbrack $
and for the corresponding angles we have (see for example \cite[Prop. 2.2.3]{Pre})
\[
z_{A}\widehat{z}w+w_{A}\widehat{w}z={\displaystyle \int \limits_{J_{zw}}%
}\left \vert \kappa_{A}\left(  s\right)  \right \vert ds .
\]
In a similar manner, for points $z_B$ (resp. $w_B$) on the line tangent to $I_B$ at the point $z$ (resp. $w$) we have
\[
z_{B}\widehat{z}w+w_{B}\widehat{w}z={\displaystyle \int \limits_{J_{zw}}}\kappa_{B}\left(  s\right)   ds .
\]
Using the last two equations, inequality (\ref{UG0}) becomes
\begin{equation}
z_{A}\widehat{z}w+w_{A}\widehat{w}z>\mathrm{\ }z_{B}\widehat{z}w+w_{B}%
\widehat{w}z.\label{UG1}
\end{equation}
Since $\left[  x,z\right]  \cup \left[  z,y_z\right]  $ is a geodesic at $z$, we have by Lemma \ref{piangle}
\begin{equation}
x\widehat{z}z_{A}+z_{B}\widehat{z}w+w\widehat{z}y_z=\pi .\label{UG2}
\end{equation}
Similarly, we have
\begin{equation}
x\widehat{w}w_{A} +w_{B}\widehat{w}z+z\widehat{w}y_w
=\pi.\label{UG3}%
\end{equation}
By summation of (\ref{UG2}) and (\ref{UG3}) and using (\ref{UG1}) we obtain
\begin{align}
 x\widehat{z}z_{A}+\underline{z_{B}\widehat{z}w}+w\widehat{z}y_z
 +
x\widehat{w}w_{A}+\underline{w_{B}\widehat{w}z}+z\widehat{w}y_w  = & \,\,2\pi \xRightarrow{\text{\scriptsize \ By (\ref{UG1}) \ }}
\nonumber \\
\left(x\widehat{z}z_{A}+\underline{z_{A}\widehat{z}w}\right) +w\widehat{z}y_z
 +
\left(x\widehat{w}w_{A}+\underline{w_{A}\widehat{w}z}\right) +z\widehat{w}y_w > &\,\, 2\pi  \nonumber
  \\
x\widehat{z}w + w\widehat{z}y_z+x\widehat{w}z + z\widehat{w}y_w > 2\pi
\end{align}
The latter inequality completes the proof of the Claim because from the triangle with vertices $x,z,w$ we have
$ x\widehat{z}w+x\widehat{w}z < \pi.$

We return now to the proof of uniqueness of geodesic segments in $N.$ Let $x,y\in N. $ We first examine the case
\[x\in \left(\Sigma_{A}\cap N\right)\setminus I_A \mathrm{\ \ and\ \ } y\in x\in \left(\Sigma_{B}\cap N\right)\setminus I_B.\] Assume there exist two
geodesic segments
\[
\lbrack x,z]\cup \lbrack z,y]\mathrm{\ and\ }[x,w]\cup \lbrack w,y]
\]
with
\[
\lbrack x,z],[x,w]\subset \Sigma_{A}\mathrm{\ and\ }[z,y],[w,y]\subset
\Sigma_{B}%
\]
and $z,w\in I_{\Sigma}\cap N$ and $z\neq w.$ By the Claim we have
\[  y\widehat{z}w+y\widehat{w}z > \pi \]
which is a contradiction in the triangle with vertices $y,z,w .$
\\[6mm]
\begin{center}
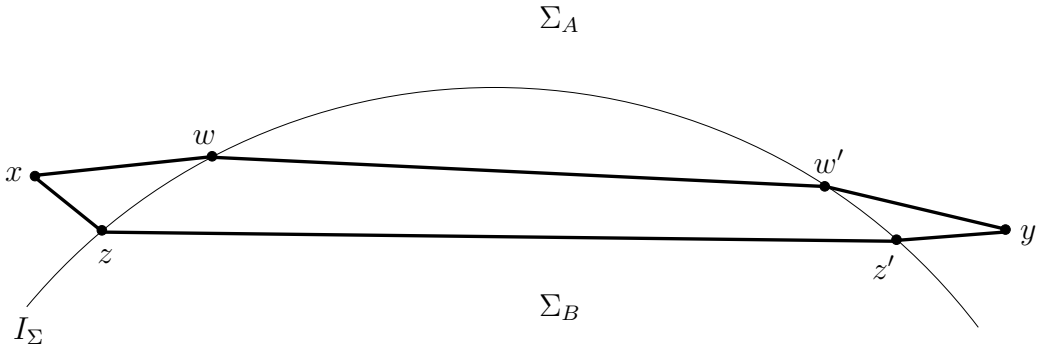
\begin{figure}[h]
\hspace*{7mm}\includegraphics[scale=0.70]{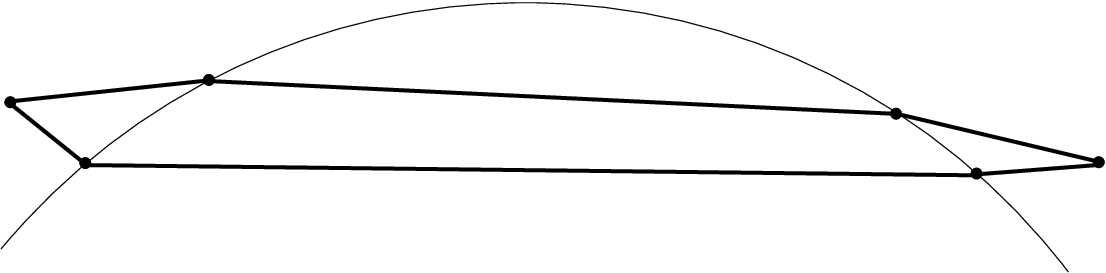}
\newline \begin{picture}(22,12)
\put(12,74){$x$}\put(396,52){$y$}
\put(318,77){$w^{\prime}$}
\put(214,23){$\Sigma_B$}
\put(83,88){$w$}
\put(15,14){$ I_{\Sigma}$}
\put(214,132){$\Sigma_A$}
\put(47,43){$ z$}
\put(340,37){$ z^{\prime}$}
\end{picture}
\caption{The case, in the proof of uniqueness in Proposition \ref{ugl}, where $x,y\in\Sigma_A$ and the assumed geodesics both intersect $\Sigma_B \setminus \partial \Sigma_B .$}
\label{n_gon2}
\end{figure}
\end{center}

We next examine the case both $x,y$ belong to $\left(\Sigma_{A}\cap N\right) \setminus I_A$ and assume that there exist two geodesics joining $x,y$
\[
\begin{split}
\lbrack x,z]\cup \lbrack z,z^{\prime}]  &  \cup \lbrack z^{\prime}%
,y]\mathrm{\  \ and\  \ }[x,w]\cup \lbrack w,w^{\prime}]\cup \lbrack w^{\prime
},y]\\
&  [z,z^{\prime}],[w,w^{\prime}]\subset \Sigma_{B}
\end{split}
\]
with $z,z^{\prime},w,w^{\prime}$ distinct points in $I_{\Sigma}\cap N$ (see Figure \ref{n_gon2}). We may also assume that $[z,z^{\prime}]\cap \lbrack w,w^{\prime}]=\emptyset$ because if
$[z,z^{\prime}]\cap \lbrack w,w^{\prime}]\neq \emptyset$ we get a contradiction from the previous case.
By applying the Claim to the geodesic segments $\lbrack x,z]\cup \lbrack z,z^{\prime}] $ and $[x,w]\cup \lbrack w,w^{\prime}]$ we obtain
$$z\widehat{w}w^{\prime} + w\widehat{z}z^{\prime} >\pi$$
and similarly,
$$w\widehat{w^{\prime}}z^{\prime} + \widehat{z^{\prime}}w^{\prime} >\pi .$$
Adding these inequalities we obtain a contradiction in the quadrilateral with vertices $w,z,z^{\prime}, w^{\prime} $ contained in $\Sigma_B \subset \mathbb{R}^2.$

We last examine the case both $x,y$ belong to $\Sigma_{A}\cap N$ and assume
that there exist two geodesics joining $x,y$ that is, one geodesic from $x$ to $y$
intersecting $I_{\Sigma}$ at two points $z, z^{\prime}$ and the other one not intersecting $\Sigma_B \setminus \partial \Sigma_B :$
\[
\lbrack x,z]\cup \lbrack z,z^{\prime}]\cup \lbrack z^{\prime}
,y]\mathrm{\  \ and\  \ }[x,y]
\]
with
\[
\lbrack x,z],[z^{\prime},y],[x,y]\subset \Sigma_{A}\mathrm{\  \ and\  \ }[z,z^{\prime}]\subset \Sigma_{B}
\]
\newline
\begin{center}
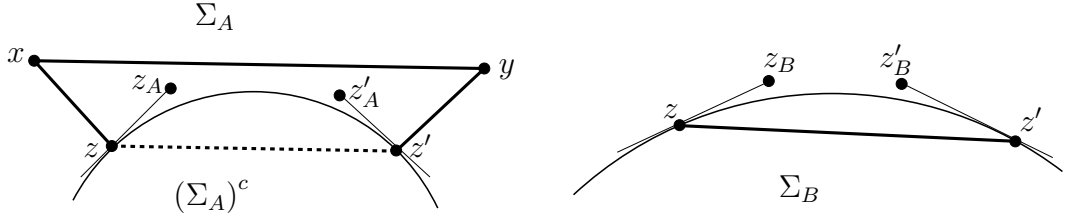
\begin{figure}[h]
\hspace*{7mm}\includegraphics[scale=0.65]{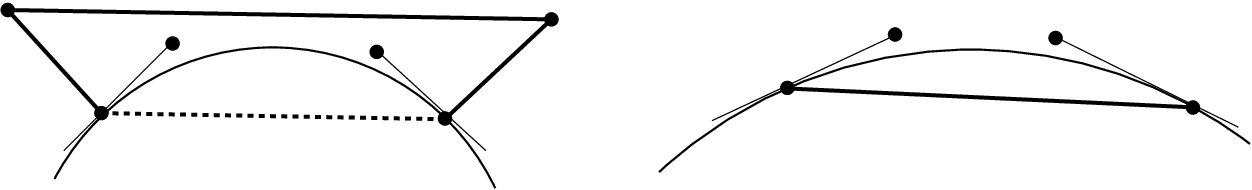}
\newline \begin{picture}(22,12)
\put(12,74){$x$}
\put(395,48){$z^{\prime}$}
\put(198,70){$y$}
\put(304,23){$\Sigma_B$}
\put(83,88){$\Sigma_A$}
\put(75,20){$ \left(\Sigma_A\right)^c$}
\put(41,38){$ z$}
\put(59,64){$ z_A$}
\put(298,72){$ z_B$}\put(342,72){$ z^{\prime}_B$}
\put(164,38){$z^{\prime}$}\put(141,60){$z^{\prime}_A$}
\put(260,53){$ z$}
\end{picture}
\caption{The case, in the proof of uniqueness in Proposition \ref{ugl}, where $x,y\in \Sigma_A$ and only one of the assumed geodesics intersects $\Sigma_B \setminus \partial \Sigma_B .$}
\label{n_gon3}
\end{figure}
\end{center}
Using the same notation (see Figure \ref{n_gon3}), as in the proof of the Claim we may use inequality (\ref{UG0}) to obtain
\begin{equation}
z_A \widehat{z}z^{\prime} +z^{\prime}_A \widehat{z^{\prime}}z
>
z_B \widehat{z}z^{\prime} +z^{\prime}_B \widehat{z^{\prime}}z.
\label{UG5}
\end{equation}
As $\lbrack x,z]\cup \lbrack z,z^{\prime}]\cup \lbrack z^{\prime},y]$ is a geodesic at $z$ we have  by Lemma \ref{piangle}
\begin{equation}
x \widehat{z}z_A +z_B \widehat{z}z^{\prime}=\pi\nonumber
\end{equation}
and similarly
\begin{equation}
y \widehat{z^{\prime}}z^{\prime}_A +z^{\prime}_B \widehat{z^{\prime}}z=\pi\nonumber
\end{equation}
By adding the above two equalities and using inequality (\ref{UG5}) we obtain
\[  x \widehat{z}z^{\prime} + y\widehat{z^{\prime}}z =
\left( x \widehat{z}z_A + \underline{z_A \widehat{z}z^{\prime}} \right) +\left(y \widehat{z^{\prime}}z^{\prime}_A
+\underline{z^{\prime}_A \widehat{z^{\prime}}z}\right)
\stackrel{\mathrm{By (\ref{UG5})}}{>}
x \widehat{z}z_A + \underline{z_B \widehat{z}z^{\prime}}
+y \widehat{z^{\prime}}z^{\prime}_A
+\underline{z^{\prime}_B \widehat{z^{\prime}}z} =2\pi.
\]
The inequality $x \widehat{z}z^{\prime} + y\widehat{z^{\prime}}z >2\pi $ is a contradiction because the quadrilateral with vertices $z,y,z^{\prime},z$ lives in $\Sigma_A \cup \left(\Sigma_A\right)^c =\mathbb{R}^2.$
Note that the above proof works verbatim in the case the segment $[x,y]$ in $\Sigma_A$ is tangent to $\partial \Sigma_A .$

This complete the proof of the Proposition.
\end{proof}
\begin{center}
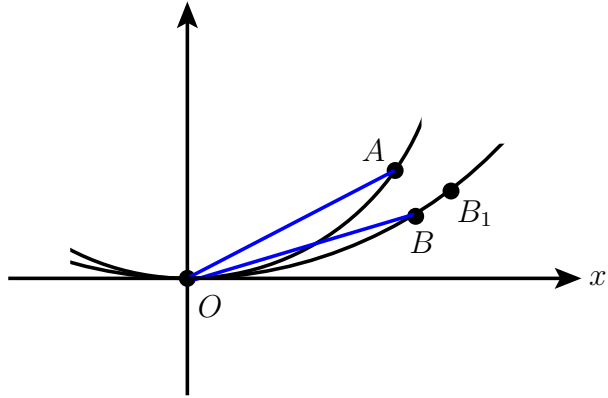
\begin{figure}[h]
\hspace*{7mm}\includegraphics[scale=3]{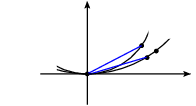}
\newline \begin{picture}(22,12)
\put(212,109){$A$}\put(230,74){$B$}
\put(298,62){$x$}\put(150,50){$ O$}
\put(248,86){$B_1$}
\end{picture}
\caption{The curvature inequality implies bigger angles.}
\label{updown}
\end{figure}
\end{center}
We will need the following

\begin{lemma}\label{downup}
Let $\sigma , \tau$ be two smooth curves in $\mathbb{R}^2$ defined on an interval $J$ containing $0$ such that
\[ \sigma (0) = \tau(0) = (0,0)
         ,
\sigma^{\prime} (0) = \tau^{\prime}(0)\]
and their curvatures $\kappa_{\sigma} ,\kappa_{\tau}$
satisfy
\[\kappa_{\sigma}(s) \leq \kappa_{\tau}(s) \mathrm{\ for\ all\ }s\in J\mathrm{\ \ and\ \ }
{\displaystyle \int \limits_{J}}\kappa_{\sigma}\left(  s\right) ds < \frac{\pi}{2},
{\displaystyle \int \limits_{J}}\kappa_{\tau}\left(  s\right) ds < \frac{\pi}{2}.
\]
Let $A,B$ be points on $\tau, \sigma $ respectively, see Figure \ref{updown},  such that the Euclidean segments $OA$ and $OB$ have equal length
\[\left\vert OA \right\vert =\left\vert OB \right \vert .\]
Then the angles formed by the segments $OA$ and $OB$ and the $x-$axis satisfy
\[ x\widehat{O}B \leq x\widehat{O}A\]
where equality holds when $\kappa_{\sigma}(s) = \kappa_{\tau}(s)$ for all $s\in J. $
\end{lemma}
\begin{proof}
By Theorem 2-19 in \cite{Gug} the lengths of the sub-arcs $\arc{OA}$ and $\arc{OB}$ of $\tau$ and $\sigma$ respectively, satisfy
\[\left\vert \arc{OA} \right\vert > \left\vert \arc{OB} \right\vert . \]
Let $B_1$ be the (unique) point on $\sigma$ such that
\[ B\in  \arc{OB}_1 \mathrm{\ \ and\ \ }\left\vert \arc{OA} \right\vert = \left\vert \arc{OB}_1 \right\vert. \]
By Theorem 2-14 in \cite{Gug} $\tau$ is above $\sigma,$ in other words, \[ x\widehat{O}B_1 < x\widehat{O}A\]
and by convexity of $\sigma$
\[ x\widehat{O}B < x\widehat{O}B_1 \]
and this completes the proof of the Lemma.
\end{proof}
We now proceed with the

\begin{proof}[Proof of Theorem \ref{main}]
We only need to check the $CAT(0)$ inequality for a ball $N$
centered at a point $w\in I_{\Sigma}.$

By assumption (k3) there exist finitely many $t_1 ,\ldots t_{\mu}\in J,$ $\mu\in\mathbb{N}\cup \{0\}$ such that $\kappa_A (t_i) +\kappa_B (t_i)=0$ for $i=1,\ldots , \mu.$ We distinguish three cases:

\underline{Case 1}: The center $w$ of $N$ is distinct from $\sigma(t_i)$ for all $i=1,\ldots , \mu$ and $w\notin \partial I_{\Sigma}.$

\underline{Case 2}: The center $w$ of $N$ coincides with $\sigma(t_{i_0})$ for some $i_0 \in \{1,\ldots , \mu\}.$

\underline{Case 3}: The center $w$ of $N$ belongs to $\partial I_{\Sigma}.$
\\
Proof of Case 1: we may assume that the closure $\overline{N}$ of $N$ does not contain  $\sigma(t_i)$ for all $i=1,\ldots , \mu.$


Let $s_w\in J$ be the unique parameter such that $w=\sigma(s_w).$ By assumption (k3) we have
\[ \left \vert \kappa_{A}\left(  s_w\right)  \right \vert -\kappa_{B}\left( s_w\right)  >0.  \]
By continuity of the curvature function, there exist a subinterval $J_{s_w}$ of $J$ containing $s_w$
and a positive $\varepsilon  $ such that
\begin{equation}
\left \vert \kappa_{A}\left(  s^{\prime}\right)  \right \vert -\kappa_{B}\left(
s^{\prime \prime}\right)  >\varepsilon
,\mathrm{\  \ for\ all\ }s^{\prime},s^{\prime \prime}\in J_{s_w}.
\label{epsilon_interval}
\end{equation}
Moreover, we may assume that
\begin{equation}
{\displaystyle \int \limits_{J_{s_w}}}\left \vert \kappa_{A}\left(  s\right)
\right \vert ds< \frac{\pi}{2}
\mathrm{\ \ and\ \ }
{\displaystyle \int \limits_{J_{s_w}}}\kappa_{B}\left(  s\right) ds < \frac{\pi}{2}
\label{UG00}%
\end{equation}

\begin{center}
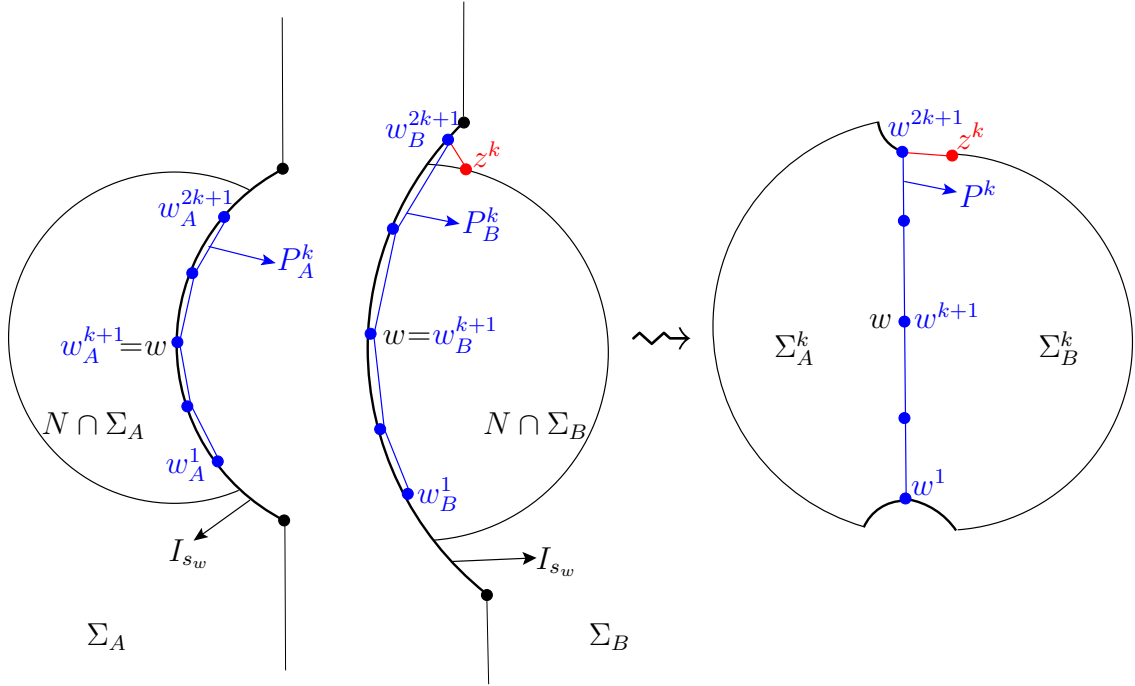
\begin{figure}[h]
\includegraphics[scale=0.5]{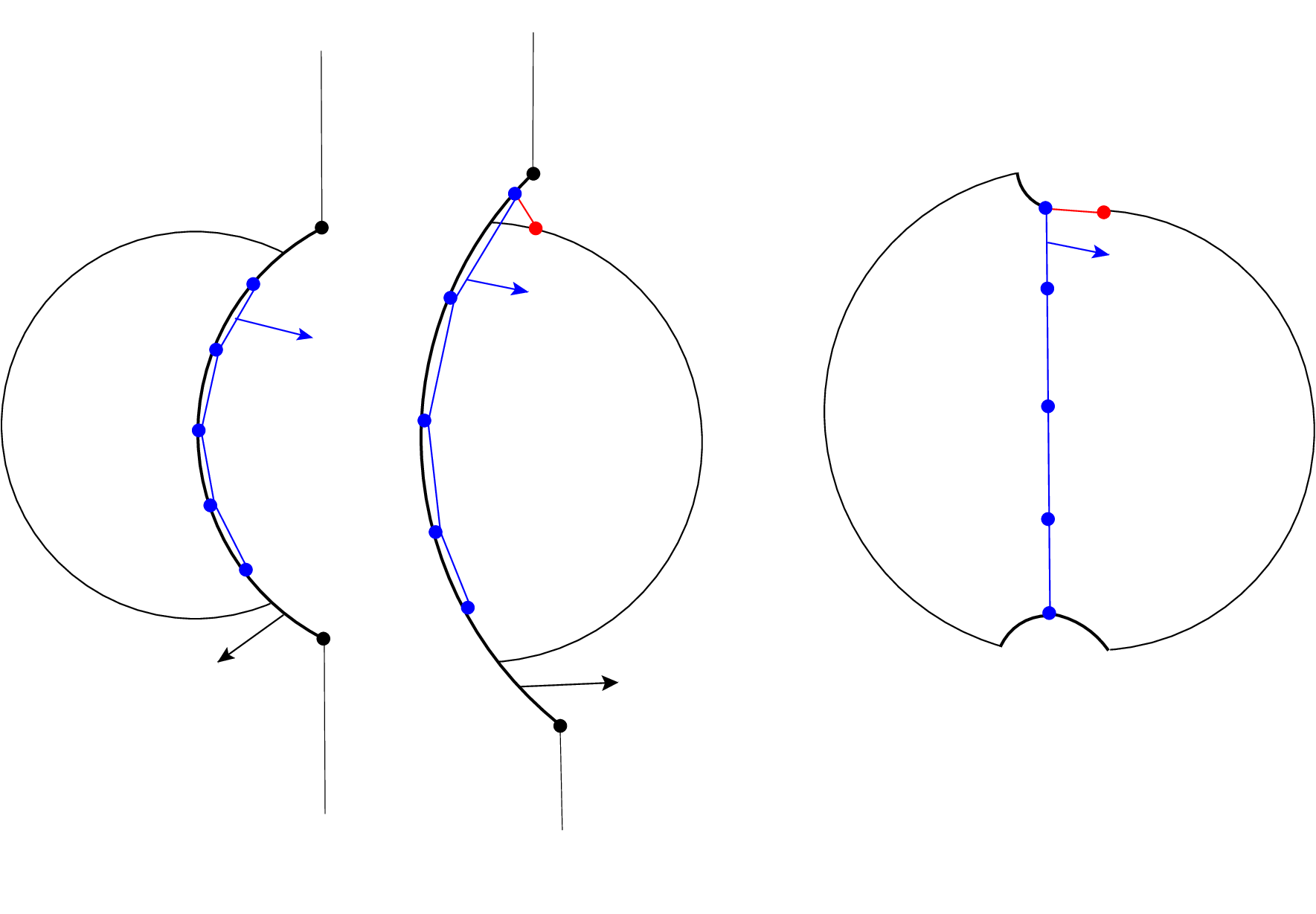}
\newline \begin{picture}(22,12)
\put(30,60){$\Sigma_A$}
\put(220,60){$\Sigma_B$}
\put(235,170){$
\mathlarger{
\mathlarger{
\mathlarger{\mathlarger{
\mathlarger{\rightsquigarrow}
           }}
           }
           }$}
\put(61,92){$I_{s_w}$}\put(200,89){$I_{s_w}$}
\put(326,180){$w$}
\put(142,174){$w\hskip-1mm = \hskip-1mm \color{blue} w^{k+1}_B$}
\put(172,216){$\color{blue} P^{k}_B$}
\put(102,202){$\color{blue} P^{k}_A$}
\put(143,252){$\color{blue} w^{2k+1}_B$}
\put(176,241){$\color{red} z^k$}
\put(359,248){$\color{red} z^k$}
\put(57,222){$\color{blue} w^{2k+1}_A$}
\put(155,114){$\color{blue} w^{1}_B$}
\put(60,125){$\color{blue} w^{1}_A$}
\put(20,170){${\color{blue} w^{k+1}_A}\hskip-2mm =\hskip-1mm w$}
\put(13,140){$N\cap \Sigma_A$}
\put(180,140){$N\cap \Sigma_B$}
\put(390,169){$\Sigma_{B}^{k}$}
\put(290,169){$\Sigma_{A}^{k}$}
\put(342,117){$\color{blue} w^{1}$}
\put(333,253){$\color{blue} w^{2k+1}$}
\put(343,180){$\color{blue} w^{k+1}$}
\put(360,226){$\color{blue} P^{k}$}
\end{picture}
\caption{The polygonal lines $P_A^k$ and $P_B^k$ and the regions $\Sigma_{A}^{k}, \Sigma_{B}^{k}$  which are glued along the polygonal lines to form $\Sigma^k .$}
\label{NN_Polygonal}
\end{figure}
\end{center}
We may further assume that the ball $N$ is small enough so that
\begin{equation}
 N\cap I_{\Sigma}\subset \sigma\left( J_{s_w}\right)\equiv I_{s_{w}}. \label{gin}
\end{equation}
Denote by $J_N$ the sub-interval of $J$ which parametrizes $ N\cap I_{\Sigma}, $ that is
\[ \sigma : J_N \longrightarrow  N\cap I_{\Sigma} . \]
By abuse of language we will denote again by $w$ the points in
$I_{A}$ and $I_{B}$ which correspond to $w\in I_{\Sigma}.$

As $\Sigma_{A}\subset \mathbb{R}^{2},$  we may form a simple polygonal line $P_{A}^{k}$  with $2k+1$ vertices $w_A^1, w_A^2, \ldots, w_A^{2k+1}$ (see Figure \ref{NN_Polygonal}) with the following properties

\begin{itemize}
\item $P_{A}^{k}$ consists of $2k$ (Euclidean) segments all of equal length $\ell_{k} ,$

\item all vertices of $P_{A}^{k}$ belong to $I_A^{k},$

\item $w$ is the  vertex $w_A^{k+1}$ of $P_{A}^{k}$ and the endpoint vertices $w_A^1, w_A^{2k+1}$ determine a sub-curve $I_{A}^{k}$ of $I_{A}$ contained in $\sigma_A \left( J_{N}\right) \equiv N\cap I_{\Sigma},$


\item $P_{A}^{k}$ minus its vertices is contained in the complement $\mathbb{R}^{2}%
\setminus \Sigma_{A}$ of $\Sigma_{A},$ (this follows from assumption (k1)),

\item as $k\rightarrow \infty,$ $\ell_{k}\rightarrow0$ and $\left\{I_{A}^{k}\right\}_{k\in\mathbb{N}}$ is an ascending sequence of sub-curves which converges to $ N\cap I_{\Sigma}.$
\end{itemize}
Clearly, if $L\left(  \cdot \right)  $ denotes the length of a curve, we have
\begin{equation}
\lim_{k\rightarrow \infty}L\left(  P_{A}^{k}\right)  =\lim_{k\rightarrow \infty
}L\left(  I_{A}^{k}\right)  =L\left(   N\cap I_{\Sigma} \right)  . \label{length_appr}%
\end{equation}
In a similar manner, we form a simple polygonal line $P_{B}^{k}$ consisting of $2k$
(Euclidean) segments all of length $\ell_{k}$ with $2k+1$ vertices $w_B^1, w_B^2, \ldots, w_B^{2k+1}$ such that

\begin{itemize}
\item $w$ is the vertex $w_B^{k+1}$ and the endpoint vertices $w_B^1, w_B^{2k+1}$ determine a sub-curve $I_{B}^{k}$ of $I_{B}$ contained in $\sigma_B \left( J_{N}\right) \equiv  N\cap I_{\Sigma},$

\item all vertices of $P_{B}^{k}$ belong to $I_B^{k},$

\item $P_{B}^{k}$ is contained in $\Sigma_{B},$ (this follows from assumption (k2)),

\item as $k\rightarrow \infty,$ $\left\{I_{B}^{k}\right\}_{k\in\mathbb{N}}$ is a (not necessarily ascending) sequence of sub-curves which converges to $ N\cap I_{\Sigma}.$
\end{itemize}
The last bullet is the only one which needs a comment: by construction, the polygonal lines $ P_{A}^{k}, P_{B}^{k}$ have equal length $L\left(  P_{A}^{k}\right)  =L\left(  P_{B}^{k}\right)  $ so  by (\ref{length_appr})
\begin{equation}
\lim_{k\rightarrow \infty}L\left(  P_{B}^{k}\right)  =\lim_{k\rightarrow \infty
}L\left(  I_{B}^{k}\right)  =L\left(   N\cap I_{\Sigma} \right)  . \label{length_appr2}%
\end{equation}
Let $\Sigma_{A}^{k}$ be the compact region bounded by
the simple closed curve
\[\left[   \partial \left( N\cap \Sigma_{A}\right)\setminus I_{A}^{k}\right]  \cup P_{A}^{k} .\]
Note that the above simple closed curve lives in $\mathbb{R}^2. $

If
\footnote[3]{ In fact, since we work in $\mathbb{R}^2,$ the inclusion $I^k_B \subset N\cap I_B $ holds: by property (\ref{epsilon_interval}) for $J_{s_{w}}$ and Theorem 2-19 in \cite{Gug}, it follows that the corresponding sub-curves $I_{A}^{k}$ and $I_{B}^{k}$
satisfy, under the identification $I_{A}\equiv I_{B},$ the inclusion $I_{B}^{k}\subset I_{A}^{k}$ which implies that $I^k_B \subset N\cap I_B .$  However, in view of the generalization in Section \ref{sec4} below, we describe what needs to be done if the inclusion $I^k_B \subset N\cap I_B $
were not valid.}
$I^k_B \subset N\cap I_{\Sigma}$ we similarly define $\Sigma_{B}^{k}$ to be the compact region bounded by the simple closed
curve $\left[   \partial \left(N\cap \Sigma_{B}\right)  \setminus I_{B}^{k}\right]  \cup P_{B}^{k}.$

If not, that is, if $w^{2k+1}_B \notin N\cap I_{\Sigma}$ (we work similarly in the case $w^{1}_B \notin N\cap I_{\Sigma}$) we may choose a sequence
$\left\{z^k \right\} \subset \partial  \left( N\cap \Sigma_{B}\right)\setminus I_{B}$ such that
\[\lim_{k\rightarrow \infty} \left\Vert z^k - w^{2k+1}  \right\Vert = 0\]
and use the segment $\left[ z^k , w^{2k+1}  \right]$ to define the region $\Sigma_{B}^{k}$ (marked in red in Figure \ref{NN_Polygonal}).

Moreover, by (\ref{length_appr2}) and (\ref{gin}) and for $k$ large enough we may assume that  $I_B^k \subset I_{s_w}$ and, hence, (see property (\ref{epsilon_interval})) we have
\begin{equation}
 \left \vert \kappa_{A}\left(  p\right)  \right \vert -\kappa_{B}\left(q \right)  >\varepsilon
,\mathrm{\  \ for\ all\ }p\in I_{A}, q\in I_B .
\label{epsilon_interval_2}
\end{equation}
Let $\Sigma^{k}$ be the space obtained by gluing $\Sigma_{A}^{k}$
with $\Sigma_{B}^{k}$ along their boundary pieces $P_{A}^{k}$ and
$P_{B}^{k}$ respectively. Let $P^{k}\subset \Sigma^{k}$ be the curve obtained
by the identification $P_{A}^{k}\equiv P_{B}^{k}.$

First observe that property (a2) of Proposition \ref{basicLemma} is satisfied, that is, for
every $x\in N$ there exists $K=K(x)\in \mathbb{N}$ such that
$x\in \Sigma^{k}$ for all $k\geq K:$ if $x\in N\cap \Sigma_{A}$ then, by property (k1) we have \[N\cap \Sigma_{A}\subset \Sigma_{A}^{k} \]
which implies that $x\in \Sigma^k $ for all $k.$ If $x\in \left(N\cap \Sigma_{B}\right)  \setminus I_{\Sigma}$ it clearly has finite distance from  $N\cap I_{B}.$ Thus, for sufficiently large $k$ or, equivalently, for sufficiently small side-length $\ell_k$ of $P_B^k$ the point $x$ belongs to $\Sigma_{B}^{k}\subset\Sigma^{k} .$

We next check property (a1) of Proposition \ref{basicLemma}, that is, $\Sigma^{k}$ is a $CAT(0)$ space. In
fact, we only need to check that the angle around each vertex $v$ of $P^{k}$
is $\geq2\pi.$ If $v_{A}$ and $v_{B}$ are the vertices in $P_{A}^{k}$ and
$P_{B}^{k}$ respectively corresponding to $v,$ denote by $\widehat{v_{A}}$ and
$\widehat{v_{B}}$ the corresponding angles in $\Sigma_{A}%
^{k}$ and $  \Sigma_{B}^{k}$ resp.
The curvature of $I_A$ viewed as the boundary of $\left( \Sigma_A\right)^c$ is positive and equal to $-\kappa_A >\kappa_B .$ The angle of  $P_A^k\subset\left(\Sigma_A\right)^c$ subtended at $v_A$ is equal to $2\pi-\widehat{v_{A}}.$ Lemma \ref{downup} implies that  $2\pi-\widehat{v_{A}}<\widehat{v_{B}},$ thus, $\widehat{v_{A}}+\widehat{v_{B}}>2\pi$ and $\Sigma^{k}$ is a $CAT(0)$ space.


\underline{CLAIM }: The sequence $\Sigma^{k}$ satisfies property (a3), that
is, for any $x,y\in N,$ $\left \vert x-y\right \vert _{k}\rightarrow$
$\left \vert x-y\right \vert $ as $k\rightarrow \infty.$

We examine in detail the case $x\in N\cap \Sigma_{A},y\in N\cap \Sigma_{B}.$ The geodesic $\left[  x,y\right]  _{k}$ intersects $P^{k}$ at a single point, say $z_{k}.$ The sequence $\left \{  z_{k}\right \}  \subset \Sigma_{B}$ must have an accumulation point $z_{0}$ which, by construction of the polygonal lines $P_A^k, P_B^k,$
necessarily belongs to $N\cap I_{\Sigma}.$ It suffices to show that $\left[
x,y\right]  \cap I_{\Sigma}=\left \{  z_{0}\right \}  .$

Suppose, on the contrary, that $\left[  x,y\right]  \cap I_{\Sigma}=\left \{
z_{0}^{\prime}\right \}  $ for some $z_{0}^{\prime}\neq z_{0}.$ By uniqueness
of geodesics, see Proposition \ref{ugl}, we must have%
\begin{equation}
\left \vert x-y\right \vert =\left \vert x-z_{0}^{\prime}\right \vert +\left \vert
z_{0}^{\prime}-y\right \vert \lneqq \left \vert x-z_{0}\right \vert +\left \vert
z_{0}-y\right \vert \label{ineq}%
\end{equation}
By construction of the sequence $\left \{  z_{k}\right \}  $ we have that
\begin{equation}
L\left(  \left[  x,z_{k}\right]  _{k}\cup \left[  z_{k},y\right]  _{k}\right)
=\left \vert x-z_{k}\right \vert _{k}+\left \vert z_{k}-y\right \vert
_{k}\rightarrow \left \vert x-z_{0}\right \vert +\left \vert z_{0}-y\right \vert
\label{ineq2}%
\end{equation}
The curve $\left[  x,z_{0}^{\prime}\right]  _{k}\cup \left[  z_{0}^{\prime
}y\right]  _{k}$ in $\Sigma^{k}$ intersects $P^{k}$ at a single point, say,
$\left(  z_{0}^{\prime}\right)  _{k}.$ Similarly, we have
\begin{equation}
L\Bigl(\left[  x,\left(  z_{0}^{\prime}\right)  _{k}\right]  _{k}\cup \left[
\left(  z_{0}^{\prime}\right)  _{k},y\right]  _{k}\Bigr)=\left \vert x-\left(
z_{0}^{\prime}\right)  _{k}\right \vert _{k}+\left \vert \left(  z_{0}^{\prime
}\right)  _{k}-y\right \vert _{k}\rightarrow \left \vert x-z_{0}^{\prime
}\right \vert +\left \vert z_{0}^{\prime}-y\right \vert \label{ineq3}%
\end{equation}
It follows that the length of the geodesic $\left[  x,y\right]  _{k}=\left[
x,z_{k}\right]  _{k}\cup \left[  z_{k},y\right]  _{k}$ approaches, as
$k\rightarrow \infty,$ the RHS of inequality (\ref{ineq}) and the length of the
curve $\left[  x,\left(  z_{0}^{\prime}\right)  _{k}\right]  _{k}\cup \left[
\left(  z_{0}^{\prime}\right)  _{k},y\right]  _{k}$ approaches the LHS of
inequality (\ref{ineq}). For some $k$ sufficiently large, this contradicts the
fact that $\left[  x,z_{k}\right]  _{k}\cup \left[  z_{k},y\right]  _{k}$ is
the geodesic $\Sigma^{k}$ joining $x$ and $y.$ The case $x,y\in N\cap
\Sigma_{A}$ and $\left[  x,y\right]  $ intersects $I_{\Sigma}$ at two points
is treated by analogous arguments. This completes the proof of CLAIM.

We may now apply Proposition \ref{basicLemma} for the sequence $\Sigma^k$ to complete the proof in Case 1 of the Theorem.

\begin{center}
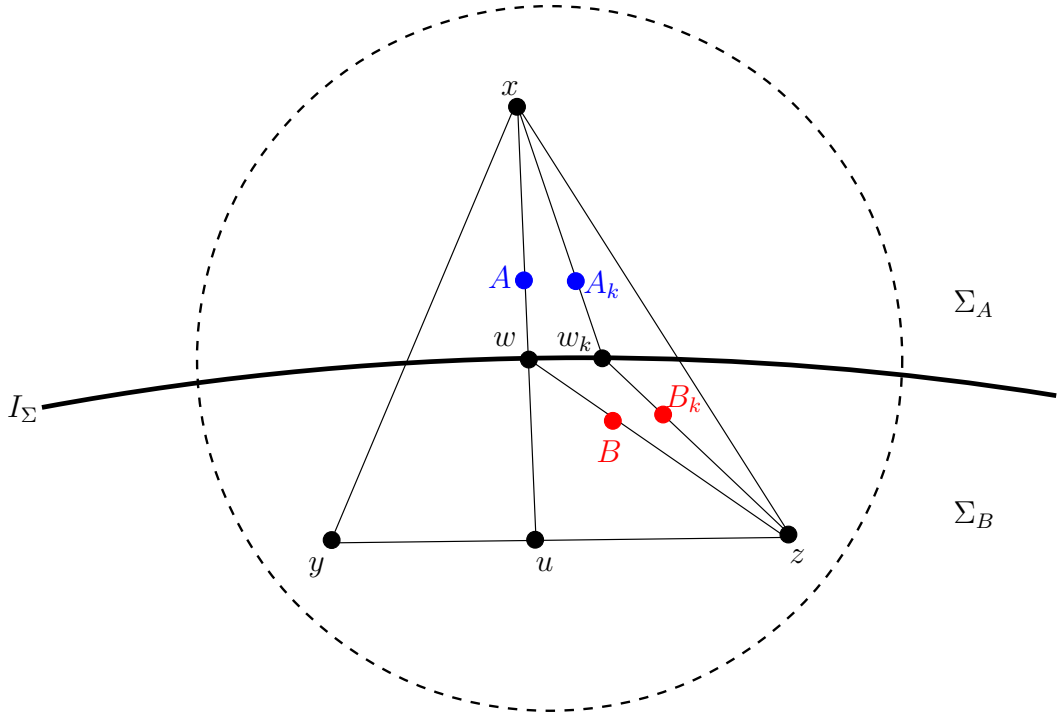
\begin{figure}[h]
\hspace*{7mm}
\includegraphics[scale=0.99]{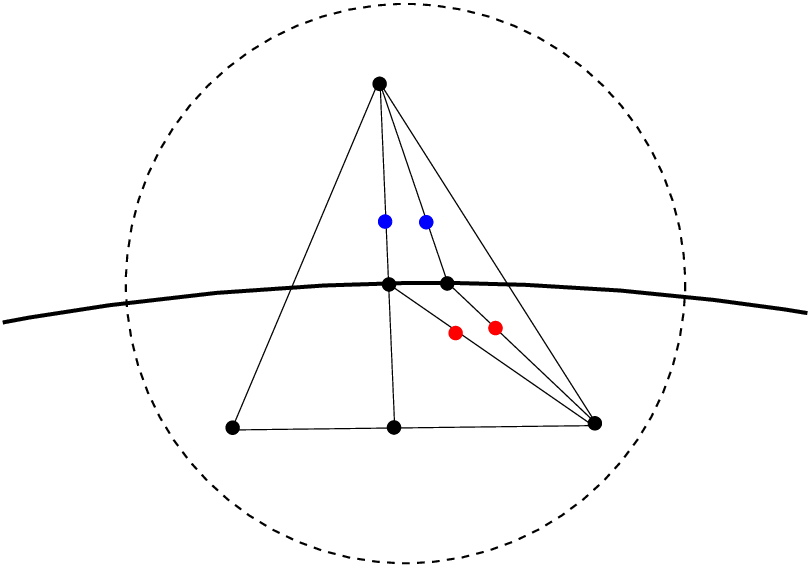}
\newline
\begin{picture}(22,12)
\put(196,157){$w$}\put(220,157){$w_k$}
\put(212,72){$u$}\put(199,252){$x$}
\put(308,75){$z$}\put(126,72){$y$}
\put(194,179){$\color{blue} A$}\put(230,177){$\color{blue} A_k$}
\put(235,113){$\color{red} B$}\put(261,134){$\color{red} B_k$}
\put(370,170){$\Sigma_A$}
\put(370,90){$\Sigma_B$}
\put(12,130){$I_{\Sigma}$}
\end{picture}
\caption{The triangles involved in the Proof of Case 2.}
\label{Fcat_last}
\end{figure}
\end{center}

\noindent \underline{Proof for Case 2:} Let the center $w$ of $N$ be a point $w=\sigma(t_{w})$ for some $t_w \in J$ such that
\[\kappa_A (t_w) +\kappa_B (t_w)=0.\]
We may assume that $w$ is the unique point in $N\cap I_{\Sigma}$ where the above equality occurs, that is, for every point $\sigma(s)\in N\cap I_{\Sigma}, s\neq t_w$ we have
\[\kappa_A (s) +\kappa_B (s)<0.\]
We will show that an arbitrary geodesic triangle $T(x,y,z)$ in $N$ is thin. If $w$ is not contained in the interior of $T\cup \mathrm{Int}(T)$ we may find a neighborhood $N_1$ containing $T$ such that for all points $\sigma(s)\in N_1\cap I_{\Sigma}$  we have
\[\kappa_A (s) +\kappa_B (s)<0.\]
Then the argument given in Case 1 applies to $N_1$ showing that $N_1$ is a $CAT(0)$ space and, thus, $T$ is thin.

Assume now that $w$ is contained in the interior of $T.$ The case $w$ belongs to one of the three sides of $T$ is a special case and is covered by the proof we give below.

Without loss of generality, we assume that $x\in N\cap\Sigma_A$ and $y,z\in N\cap \Sigma_B .$ As $w$ is the interior of $T,$ the extension of the geodesic segment $[x,w]$ must intersect $[y,z]$ at a point, say, $u.$

It suffices to show that the triangle $T\left( x,w,z \right)$ is a $CAT(0)$ space. Assuming this, observe that the triangle
$T\left( u,w,z \right)$ is Euclidean and thus the gluing of $T\left( x,w,z \right)$ with $T\left( u,w,z \right)$ along $[w,z]$ is again a $CAT(0)$ space. That is, the triangle $T\left( x,u,z \right)$ is a $CAT(0)$ space. Similarly, the triangle $T\left( x,u,y \right)$ is $CAT(0)$ and so is $T(x,y,z)$ because
\[T(x,y,z) =T\left( x,u,y \right) \cup_{[x,u]} T\left( x,u,z \right) \]
We conclude the proof of Case 2 of the Theorem by showing that the triangle with vertices $ x,w,z$ is thin.

Assume, on the contrary, that there exist points $A\in [x,w], B\in[w,z]$ such that for the comparison triangle with vertices $\overline{x},\overline{w},\overline{z}$ we have
\begin{equation}
 \left| A-B\right| > \left \Vert \overline{A} -\overline{B} \right \Vert \label{last_cat1}
 \end{equation}
where $\overline{A} , \overline{B}$ are the corresponding points for $A,B$ respectively (the case $A\in [x,w], B\in[x,z]$ is treated similarly).

Pick a sequence of points $\left\{ w_k\right\}_{k\in\mathbb{N}}$ with the properties
\[ \left\{ w_k \right\}\subset T\left( x,w,z \right)\cap I_{\Sigma} \mathrm{\ \ and\ \ } w_k\longrightarrow w,
w_k\neq w.\]

For each $k,$ form the triangle $T_k\equiv T_k\left( x,w_k,z \right)$ which does not contain $w$ and hence, as explained above, $T_k$ is thin. Let $A_k$ (resp. $B_k$) be the unique point on the segment $\left[ x,w_k \right]$ (resp. $\left[ w_k ,z \right]$) with the property $\left| w_k-A_k\right| =\left| w-A\right|  $
$\big($resp. $\left| w_k - B_k\right| =\left| w-B\right|  \big).$
Moreover, $A_k\rightarrow A$ and $ B_k\rightarrow B,$ hence
\begin{equation}
\left| A_k-B_k\right| \longrightarrow\left| A-B\right| . \label{last_cat2}
\end{equation}
As a Euclidean triangle is determined by the lengths of its sides, it is easy to see that
\begin{equation}
 \Vert \overline{A}_k -\overline{B}_k \Vert\longrightarrow\Vert \overline{A}-\overline{B} \Vert \label{last_cat3}
\end{equation}
For large enough $k,$ inequality (\ref{last_cat1}) along with (\ref{last_cat2}) and (\ref{last_cat3}) imply
\[  \left| A_k-B_k\right| > \left \Vert \overline{A}_k -\overline{B}_k \right \Vert  \]
which contradicts the fact that the triangle $T_k$ is a thin.\\
This completes the proof in Case 2.

\noindent \underline{Proof for Case 3:} In this case the proof is almost identical with the proof given in Case 2: consider an arbitrary geodesic triangle $T\equiv T(x,y,z).$ Clearly, $w$ cannot belong to the interior of $T,$ thus, either $w$ belongs to a side of $T$ or, we may find a neighborhood $N_1$ containing $T$ such that for all points $\sigma(s)\in N_1\cap I_{\Sigma}$  we have
\[\kappa_A (s) +\kappa_B (s)<0.\]
In the latter case the proof given in Case 1 suffices and in the former case the proof that $T$ is thin is identical with the proof, given in Case 2, that the triangle $T\left( x,u,z \right) $ is thin.
\end{proof}
In the following example we exhibit the necessity of assumption \textbf{(k3)} for the validity of Theorem \ref{main}.
\begin{center}
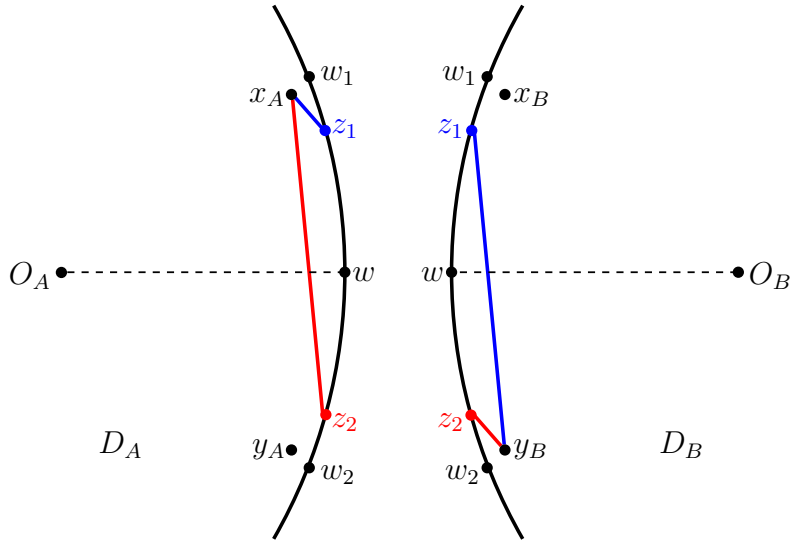
\begin{figure}[h]
\hspace*{23mm}
\includegraphics[scale=1.49]{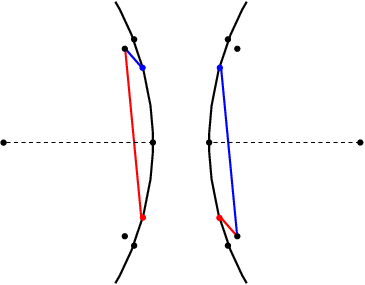}
\newline
\begin{picture}(22,12)
\put(243,185){$x_B$}\put(143,185){$x_A$}
\put(144,52){$y_A$}\put(243,52){$y_B$}
\put(52,116){$O_A$}\put(332,116){$O_B$}
\put(86,52){$D_A$}\put(298,52){$D_B$}
\put(182,117){$w$}\put(208,117){$w$}
\put(170,193){$w_1$}\put(216,193){$w_1$}
\put(170,42){$w_2$}\put(217,42){$w_2$}
\put(174,173){$\color{blue} z_1$}\put(214,173){$\color{blue} z_1$}
\put(214,62){$\color{red} z_2$}\put(174,62){$\color{red} z_2$}
\end{picture}
\caption{Notation of the Example.}
\label{figex}
\end{figure}
\end{center}
\textbf{Example.}
Consider two copies $D_A,D_B$ of an Euclidean disk of radius $R.$ Clearly $\kappa_A + \kappa_B >0$ at every boundary point and by choosing $R$ large enough we can make the  sum $\kappa_A + \kappa_B $ arbitrarily close to $0.$ Let $I_A$ and $I_B$ be subsets of $\partial D_A$ and $\partial D_B$ respectively such that  $I_A$ and $I_B$ have the same length. As above,
we may glue $D_{A}$ with $D_{B}$ along their boundary pieces $I_{A}\equiv I_{B}$ to form a connected surface $\Sigma$
\[
\Sigma:= D_{A}\cup_{I_{A}\equiv I_{B}}D_{B}.
\]
It is easy to see that $\Sigma$ is not a $CAT(0)$ space. In fact, if $O_A,O_B$ are the centers of the disks $D_A,D_B$ respectively, then for any point $w\in I$ the union
\[ [O_A , w]\cup [w, O_B] \]
is a geodesic from $O_A$ to $O_B.$

Moreover, $\Sigma$ is not even locally uniquely geodesic. To see this, let $\cal{U}$ be an arbitrarily small neighborhood of a point $w\in I. $ Pick points $w_1,w_2 \in I\cap \cal{U}$ such that $w,w_1,w_2$ form an isosceles triangle with $\left\vert w-w_1\right\vert = \left\vert w-w_2\right\vert.$ Let $\varepsilon$ be the positive number such that
\begin{equation}
\left\vert w-w_1\right\vert + \left\vert w-w_2\right\vert=\varepsilon + \left\vert w_1-w_2\right\vert .\label{examplee}
\end{equation}
Pick and fix points $x_A,y_A\in D_A\cap\cal{U}$ which are symmetric with respect to the ray $[O_A ,w]$ such that
\begin{equation}
\left\vert x_A-w_1\right\vert = \frac{\varepsilon}{4} =\left\vert y_A-w_2\right\vert .\label{examplee2}
\end{equation}
All the above notation is gathered in Figure \ref{figex} above.
Define the continuous function
\[f:I\rightarrow \mathbb{R} : f(z) =\left\vert x_A-z\right\vert +\left\vert z- y_A\right\vert .\]
By (\ref{examplee},\ref{examplee2}), the value $f(w)=\left\vert w-w_1\right\vert + \left\vert w-w_2\right\vert$ is strictly larger than both $f(w_1)$ and $f(w_2).$ It follows that $f$ attains its minimum at a point $z_1 \neq w$ and by symmetry there also exists a point $z_2$ such that the number
\[ \left\vert x_A-z_1\right\vert + \left\vert z_1-y_A\right\vert
= \left\vert x_A-z_2\right\vert + \left\vert z_2-y_A\right\vert \]
is the minimum amongst the lengths of all piece-wise geodesics $[x_A ,z]\cup [z,y_A]$ where $z\in I.$

Consider the corresponding points $x_B, y_B \in D_B .$ Then for the points $x_A , y_B\in \Sigma$ there exists two distinct geodesics joining them, namely,
\[ [x_A, z_1] \cup [z_1,y_B]  \textrm{\ \ and\ \ } [x_A, z_2] \cup [z_2,y_B].\]

\section{Generalization to Riemann Surfaces\label{sec4}}

Instead of $\mathbb{R}^{2}$ we consider a simply connected Riemannian
surface $X$ of class $C^{\infty }$ and of curvature $k\leq 0.$ We also
assume that $X$ is geodesically complete. By Hopf-Rinow theorem the geodesic
completeness of $X$ is equivalent to the fact that $X$ is complete as metric
space and to the fact that each closed and bounded subset $X$ is compact.

In what follows by $\nabla $ will denote the Riemannian connection and by $%
\langle \cdot $ $,\cdot \rangle $ the Riemannian metric of $X.$

Let $r(s),$ $s\in (a,b)$ be a smooth curve of $X$ i.e. of class $C^{\infty }$
parametrized by arc length. Let $T(s)=r^{\prime }(s)$ be the unit tangent
vector of $r.$ The vector field $\left(\nabla _{T}T\right)(s)$ is normal to $T(s)$ and
the curvature $k(s)$ of $r$ at the point $r(s)$ is defined by
\[
k(s)=\left\vert\left(\nabla _{T}T\right)(s)\right\vert=\sqrt{\langle
\left(\nabla _{T}T\right)(s), \left(\nabla _{T}T\right)(s)
\rangle }.
\]

Let $\Sigma $ be a $2-$dimensional connected complete sub-manifold of $X$
whose boundary $\partial \Sigma $ consists of finitely many components each
being a piece-wise smooth curve in $X.$ Equip $\Sigma $ with the induced
from $X$ length metric. The topology with respect to the induced length
metric is equivalent to the relative topology from $X.$ This follows from
the fact that any two points in $\partial \Sigma $ have finite distance. It
follows that $\Sigma $ with the induced length metric is complete and
locally compact, hence, a geodesic metric space. We will be calling such a
space a \emph{domain} in $X.$

Below we need the following which is Theorem 1.3 in \cite{LyWe}:
\begin{theorema}
 Let $Z$ be a geodesic metric space homeomorphic to the closed disk. Denote by $\partial Z$ the boundary circle and assume that $Z\setminus \partial Z$ is a locally $CAT(0)$ space. Then the following are equivalent:\\
 (1) $Z$ is $CAT(0)$\\
 (2) $Z\setminus \partial Z$ with the metric induced from $Z$ is a length space.
\end{theorema}

\begin{proposition}
A domain $\Sigma $ in $X$ is a locally $CAT(0)$ space.
\end{proposition}
\begin{proof}
 For any $x\in \Sigma\setminus\partial\Sigma$ we can clearly find a ball $V\subset X$ centered at $x$ with $V\subset \Sigma\setminus\partial\Sigma . $ Then $V$ is $CAT(0)$ because $X$ is of curvature $k\leq 0.$

 If $x\in\partial\Sigma$ then for a ball $V\subset X$ centered at $x,$ Theorem A above applies to $Z=V \cap \Sigma.$
\end{proof}
\subsection{The sign curvature and its geometric meaning\label{scurv}}

For $s\in (a,b)$ a positive number $\delta _{s}>0$ can be defined such that:
if $t\in (-\delta _{s},\delta _{s})$ then $s+t\in (a,b).$ Subsequently, for
a fixed $s,$  an auxiliary curve in the tangent space $T_{r(s)}X$
$$\gamma :(-\delta _{s},\delta_{s})\rightarrow T_{r(s)}X$$
can be defined as follows:
\[
\gamma (t)=\exp ^{-1}(r(s+t)),\text{ }t\in (-\delta _{s},\delta _{s}).
\]%
Let $T_{s}$ be the tangent vector to $r$ at the point $r(s).$ In Lemma 3.1
of \cite{LMM} it is proven that
\[
\gamma ^{\prime }(0)=T_{s}\text{ and }\gamma ^{\prime \prime }(0)=(\nabla
_{T}T)_{s}.
\]

We assume now that the smooth curve $r$  appears as the boundary (or as a
piece of the boundary) of a domain $\Sigma $ in $X$ and let $r(s),$ $s\in
(a,b)$ be a parametrization of $r$ by arc length. Then we may define
naturally the sign curvature of $r$ as follows: to each point $r(s)$ we
consider the unit normal vector $n(s)$ to the curve $r$ which is directed
inward the domain $\Sigma .$ Then it will be
\[
(\nabla _{T}T)_{s}=\overline{k}(s)n(s)
\]%
where $\overline{k}(s)$ is a function defined for each $s\in (a,b)$ and
which can take any real value. The function $\overline{k}(s)$ will be referred to as signed curvature of $r.$

We will show the following result which is well-known in the case of $%
\mathbb{R}^{2}$ equipped with its usual inner product and connexion.

\begin{proposition}\label{innersegment}
Let $r$ be a piece of the boundary of a domain $\Sigma $ in $X$ and let $r(s),$ $s\in (a,b)$ be a parametrization of $r$ by arc length. Let $s_{0}\in
(a,b)$ such that the signed curvature $\overline{k}(s_{0})>0.$ Then we may
find a neighborhood $(s_{0}-\varepsilon ,s_{0}+\varepsilon ),$ $\varepsilon
>0$ of $s_{0}$ in $(a,b)$ such that for every two points $s_{1},$ $s_{2}$ in
$(s_{0}-\varepsilon ,s_{0}+\varepsilon )$ the unique geodesic of $X$
joining $r(s_{1})$ and $r(s_{2})$ is lying in $\Sigma .$
\end{proposition}

\begin{proof}
Since $\overline{k}(s_{0})>0$ it follows that the vectors $n(s_{0})$ and $%
(\nabla _{T}T)_{s_{0}}$ are pointing in the same direction. To prove our
result it suffices to prove that the angle $\measuredangle (n(s_{0}),$ $%
\gamma (t)),$ for $t\in (0,\epsilon (s_{0})),$ it is smaller than $\pi /2$
for a positive $\epsilon (s_{0})$ sufficiently small. By construction the
curve $\gamma (t)$ in $\mathbb{R}^{2}$ is convex in a neighborhood of $0.$
Therefore $\measuredangle (\gamma ^{\prime \prime }(0),\gamma (t))<\pi /2$
for positive $t$ belonging in a neighborhood of $0.$ Therefore, since $%
\gamma ^{\prime \prime }(0)=(\nabla _{T}T)_{s_{0}}$ and $n(s_{0})$ and $%
(\nabla _{T}T)_{s_{0}}$ are pointing in the same direction our result
follows.
\end{proof}

\subsection{The main result}

By means of the curve $\gamma $ defined above the following function $U(t)$
is defined:
\[U(t)= \begin{cases} \displaystyle
\frac{-\gamma (t)}{\left\Vert\gamma (t)\right\Vert}, & \text{if }t<0\\[4mm]
          T_{s}, & \text{ if }t=0 \\[4mm]
   \displaystyle       \frac{\gamma (t)}{\left\Vert\gamma (t)\right\Vert}, &\text{ if }t>0
		 \end{cases}  \]
Subsequently the following angles
\begin{align}
& \theta_{r(s)}^{+}(t)=\measuredangle (U(0),U(t))\textrm{ if } t>0
 \nonumber \\
 \textrm{and\ } &
\theta _{r(s)}^{-}(t)=\measuredangle (U(0),U(t))\textrm{ if } t<0
\nonumber
\end{align}
are defined. Obviously $\theta _{r(s)}^{+}(0)=\theta
_{r(s)}^{-}(0)=0.$ The following result is proven in Lemma 3.3. of \cite{LMM}%
.

\begin{lemma}
\label{basic1}(1) The function $\theta _{r(s)}^{+}(t),$ $t\in \lbrack 0,\delta _{s})$ is increasing and
\[\frac{d\theta _{r(s)}^{+} (t)}{dt}(0)=\frac{1}{2}k(s).\]
(2) The function $\theta _{r(s)}^{-}(t),$ $t\in (-\delta _{s},0]$ is decreasing and
\[\frac{d\theta _{r(s)}^{-} (t)}{dt}(0)=-\frac{1}{2}k(s)\]
\end{lemma}
We also define the angle $\phi_{r(s)} (t)=\measuredangle (U(-t),U(t))$ if $t\in [0,\delta _{s})$ and we have

\begin{lemma}
\label{basic2}$\displaystyle
\frac{d\phi _{r(s) (t)}}{dt}(0)=k(s).$
\end{lemma}

\begin{proof}
We have that
\[ \phi _{r(s)}(t)=\theta _{r(s)}^{+}(t)+\theta _{r(s)}^{-}(-t)
 \textrm{ for }t\in  [0,\delta _{s})
\]
hence, by Lemma \ref{basic1}
\begin{align} \frac{d\phi _{r(s) (t)}}{dt}(0) &\nonumber
 =\frac{d\theta _{r(s)}^{+} (t)}{dt}(0) +
  \frac{d\theta _{r(s)}^{-} (t)}{dt}(0) \\
  & = \frac{1}{2}k(s) - \left( -\frac{1}{2}k(s)\right)= k(s).\nonumber
\end{align}
\end{proof}

As a consequence of the previous lemmata we have the following.

\begin{corollary}\label{102Cor}
Let $r_{1}(s),$ $r_{2}(s),$ $s\in (a,b)$ be two curves in $X$ parametrized
by arc length. Let $k_{1}(s),$ $k_{2}(s)$ be the curvature functions of $%
r_{1},$ $r_{2}$ respectively and let $k_{1}(s_{0})>k_{2}(s_{0})$ for some $%
s_{0}\in (a,b).$ Then,

(1) there is $\delta >0$ such that $\theta _{r_{1}(s_{0})}^{+}(t)>\theta
_{r_{2}(s_{0})}^{+}(t)$ for each $t\in \lbrack 0,\delta ).$

(2) there is $\delta >0$ such that $\phi _{r_{1}(s_{0})}(t)>\phi
_{r_{2}(s_{0})}(t)$ for each $t\in \lbrack 0,\delta ).$
\end{corollary}

\begin{proof}
(1) $k_{1}(s_{0})>k_{2}(s_{0})$ implies that $\frac{d\theta _{r_{1}(s)}^{+}}{%
dt}(0)>\frac{d\theta _{r_{2}(s)}^{+}}{dt}(0)>0$ from Lemma \ref{basic1}.
Therefore there is $\delta >0$ such that $\frac{d\theta _{r_{1}(s)}^{+}}{dt}%
(t)>\frac{d\theta _{r_{2}(s)}^{+}}{dt}(t)>0$ for each $t\in \lbrack 0,\delta
).$ Now $\theta _{r_{1}(s)}^{+}(0)=\theta _{r_{2}(s)}^{+}(0)=0$ and both $%
\theta _{r_{1}(s)}^{+}(t),$ and $\theta _{r_{2}(s)}^{+}(t)$ are increasing
for $t\in \lbrack 0,\delta ).$ Therefore $\theta _{r_{1}(s)}^{+}(t)>\theta
_{r_{2}(s)}^{+}(t)$ for each $t\in \lbrack 0,\delta ).$

Statement (2) is similarly proven if instead of $\theta _{r(s_{0})}^{+}$
we use the function $\phi _{r(s_{0})}.$
\end{proof}

\subsection{Gluing Domains from a Riemann Surface\label{subsec4}}
Let  $X$ be a simply connected, geodesically complete Riemannian surface of class $C^{\infty }$ and of curvature $k\leq 0.$

Let $\Sigma_{A},\Sigma_{B}$ be domains in $X$ as described above. We will use the exact same notation introduced in Section \ref{sec1} after Proposition \ref{cat0subsets}, that is,
$I_{A},$ $I_{B}$ are closed (finite) subintervals of $\partial
\Sigma_{A},$ $\partial \Sigma_{B}$ respectively both parametrized by arc-length by the same real interval $J$
\[
\sigma_{A}:J\rightarrow I_{A}\mathrm{\ and\ }\sigma_{B}:J\rightarrow I_{B}.
\]
For each $s\in J$ denote by $\overline{\kappa}_{A}\left(  s\right)  $ (resp. $\overline{\kappa}_{B}\left(  s\right)  $) the signed curvature of $I_{A}$ (resp. $I_{B}$) at
the point $\sigma_{A}\left(  s\right)  $ (resp. $\sigma_{B}\left(  s\right)$) as defined in Subsection \ref{scurv} above. Assume the same three properties for all $s\in J$ as in the case with domains from $\mathbb{R}^2 :$
\begin{description}
\item[(k1)] $\overline{\kappa}_{A}\left(  s\right)  \leq 0$ for all $s\in J.$
\item[(k2)] $\overline{\kappa}_{B}\left(  s\right) \geq 0$ for all $s\in J.$
\item[(k3)] $\overline{\kappa}_{A}\left(  s\right)  +\overline{\kappa}_{B}\left(  s\right)  \leq 0$ where equality holds for finitely many points in $J.$
\end{description}
and let $\Sigma$ be the result of gluing $\Sigma_{A}$ with $\Sigma_{B}$ along their isometric boundary pieces $I_{A}\equiv I_{B}$
\[
\Sigma:= \Sigma_{A}\cup_{I_{A}\equiv I_{B}}\Sigma_{B}
\]
The surface $\Sigma$ is connected with piece-wise smooth boundary
\[\partial \Sigma=\left(  \partial \Sigma
_{A}\backslash \mathrm{Int}\left(  I_{A}\right)  \right)  \cup \left(\partial \Sigma_{B}\backslash \mathrm{Int}\left(  I_{B}\right)  \right)  \]

As $\Sigma$ is a locally compact, complete length space, $\Sigma $ is a geodesic metric space. Following closely the line of proof of Theorem \ref{main} we will show the following

\begin{theorem}
\label{mainriemann}The surface $\Sigma$ is a locally $CAT(0)$ metric space.
\end{theorem}
We first show local uniqueness of geodesics in $\Sigma .$
\begin{proposition}
\label{uglreimann}Geodesic segments in $\Sigma$ are locally unique.
\end{proposition}
For the proof of the above Proposition we will need the following
\begin{claim}\label{piangleriemann}
Lemma \ref{piangle} proved in Section 3 holds in the general case of Riemann surfaces.
\end{claim}
\begin{proof} Using the existence of isothermal coordinates around each point of the surface $X$ (see for example \cite[Section 9.5]{Hic}), we may
 consider chart $\left( \cal{U}_A ,\phi_A \right)$ around $z_A$ such that (see for example \cite[Section 9.5]{Hic})
 \begin{itemize}
  \item $\phi_A \left( z_A \right)$ is the origin $O_A$ in $\mathbb{R}^2$
  \item the metric $d_A$ at each point $q\in \phi_A \left( \cal{U}_A \right)$ has the form $f_A(q)\left( dx^2+dy^2 \right)$ for some smooth function $f_A$ with $f_A \left( O_A \right)=1$
  \item if $\sigma_A (0)=z_A$ where $\sigma_A$ is the arc-length parametrization of $I_A$ then $\sigma_A^{\prime} (0)$ is mapped, under the differential $d\phi_A$ to the vector $(1,0).$
 \end{itemize}
By abuse of language we will  write again $\sigma_A $ for the arc-length parametrization of $\phi_A \left( \cal{U}_A \cap I_A \right).$

Similarly, consider chart $\left( \cal{U}_B ,\phi_B \right)$ around $z_B$ with the same as above properties. Without loss of generality we may assume that $[x,z_A]\subset \cal{U}_A$ and $[y,z_B]\subset \cal{U}_B.$

Using the notation introduced in the proof of Lemma \ref{piangle}, it suffices to show that the limit as $s\rightarrow 0$ of the following ratio is $\lvertneqq 1 :$
\begin{equation}
 \frac{d_A\left( \alpha (s) ,\sigma_A(\rho s)\right) +
          d_B\left( \beta (s) , \sigma_B(\rho s)\right)
}
{d_A\left( \alpha (s) , z_A \right) +d_B \left( \beta (s) , z_B \right)}
 \label{ratio}
\end{equation}
Let $\varepsilon_0>0$ be a positive integer such that
\begin{equation}
\bigl(\left\vert AE \right\vert + \left\vert EB \right\vert\bigr)
\left( 1+ \varepsilon_0\right)<
\left\vert OA \right\vert +\left\vert OB \right\vert
\label{laste}
\end{equation}
where $OA,OB$ and $AEB$ are the sides of the triangle defined in Observation \ref{obs1}.

As the metric at $O_A$ is Euclidean (recall $f_A \left( O_A \right)=1$) and $f_A$ is smooth,  we have that for all sufficiently small $s$
\[ d_A\left( \alpha (s) ,\sigma_A(\rho s)\right) <
    \left\vert  \alpha (s) - \sigma_A(\rho s)\right\vert \left( 1+ \varepsilon_0\right) \]
where $  \left\vert \,\,\,\, \right\vert  $ in the above inequality denotes Euclidean distance as usual. Similarly, for all sufficiently small $s$
\[ d_B\left( \beta (s) ,\sigma_B(\rho s)\right) <
    \left\vert  \beta (s) - \sigma_B(\rho s)\right\vert \left( 1+ \varepsilon_0\right) .\]
It follows that the ratio in (\ref{ratio}) is
\begin{align}
 \leq &\,\,\, \frac{
 \left\vert \alpha (s) - \sigma_A(\rho s)\right\vert +
          \left\vert \beta (s) - \sigma_B(\rho s)\right\vert
}   {2s}\left( 1+ \varepsilon_0\right)\nonumber
 \\[4mm]
= & \,\,\,
  \frac{1}{2}
  \left(\left\vert \frac{\alpha (s)}{s} - \frac{\sigma_A(\rho s)}{s}\right\vert +
  \left\vert \frac{\beta (s)}{s} - \frac{\sigma_B(\rho s)}{s}\right\vert
  \right) \left( 1+ \varepsilon_0\right)
\label{parast}
\end{align}
As in the proof of Lemma \ref{piangle}, the limit as $s\rightarrow 0$ of (\ref{parast}) is equal to
\begin{align}
\frac{1}{2} \Bigl(
 \left\vert \alpha^{\prime} (0) -\sigma_{A}^{\prime} (0) \rho \right\vert +
\left\vert \beta^{\prime} (0) -\sigma_{B}^{\prime} (0) \rho \right\vert
\Bigr) \left( 1+ \varepsilon_0\right) &=
\frac{
\left\vert AE \right\vert + \left\vert EB \right\vert}{2}
\left( 1+ \varepsilon_0\right)\nonumber \\
\Bigl[\mathrm{by}\,\,\,(\ref{laste})\Bigr] \,\,\,\, \,\,& <
\frac{\left\vert OA \right\vert +\left\vert OB \right\vert}{2} =1.\nonumber
\end{align}

\end{proof}
\begin{proof}[Proof of Proposition \ref{uglreimann}]
 The Claim presented in the proof of Proposition \ref{ugl} holds in the Riemannian setup. To see this observe, using the exact same notation introduced in Figure \ref{n_gon}, that
 \begin{equation}
z_{A}\widehat{z}w> z_{B}\widehat{z}w \mathrm{\ and\ } w_{A}\widehat{w}z>w_{B}
\widehat{w}z.\nonumber
\end{equation}
These two inequalities follow from Lemma \ref{basic1}(1),(2) and thus
\begin{equation}
z_{A}\widehat{z}w+w_{A}\widehat{w}z > z_B\widehat{z}w  +w_B\widehat{w}z.\label{UG1r}
\end{equation}
which is identical with equation (\ref{UG1}). Using equations (\ref{UG2}) and (\ref{UG3}) which, by Lemma \ref{piangleriemann}, hold verbatim in the Riemannian setup we obtain in the same manner
\[
x\widehat{z}w + w\widehat{z}y_z+x\widehat{w}z + z\widehat{w}y_w > 2\pi.
\]
The triangle with vertices $x,z,w$ lies in the surface $X$ which has curvature $\leq 0$ and, hence, $ x\widehat{z}w+x\widehat{w}z < \pi$ which implies
\[w\widehat{z}y_z+ z\widehat{w}y_w >2\pi\]
which is the statement of the Claim. \\
The rest of the proof the Proposition is identical with the proof given in Proposition \ref{ugl}.
\end{proof}
\begin{proof}[Proof of Theorem \ref{mainriemann}]
The proof uses Proposition \ref{basicLemma} in an identical way as in the proof of Theorem \ref{main}. The sequence of $CAT(0)$ spaces satisfying properties (a1)-(a3) of Proposition \ref{basicLemma} is constructed using piece-wise geodesic  segments $P_A^k ,P_B^k, k\in\mathbb{N}$ in $X$ instead of Euclidean polygonal lines. Proposition \ref{innersegment} asserts that
\begin{equation}P_A^k \subset\left(  \Sigma_{A}\right)^c\cup I_A \mathrm{\ and\ } P_B^k\subset\Sigma_B.\label{curv_2riem} \end{equation}
Using the same notation as in the proof of Case 1 of Theorem \ref{main} 
we work in a sufficiently small neighborhood $N$ of a point $w\in I_{\Sigma}$ so that
\begin{equation}
 \left \vert \overline{\kappa}_{A}\left(  p\right)  \right \vert -\kappa_{B}\left(q \right)  >\varepsilon
,\mathrm{\  \ for\ all\ }p\in I_{A}, q\in I_B .
\label{epsilon_interval_2riem}
\end{equation}
The compact regions $\Sigma^k_{A}, \Sigma^k_{B}$ are similarly defined and we set
\[ \Sigma^k = \Sigma^k_{A}\cup_{P_A^k\equiv P_B^k}\Sigma^k_{B}.
\]
By (\ref{epsilon_interval_2riem}) $\Sigma^k_{A}\subset \Sigma_A \cup N$ for all $k$ and since
\begin{equation}
\lim_{k\rightarrow \infty}L\left(  P_{B}^{k}\right)  =\lim_{k\rightarrow \infty
}L\left(  I_{B}^{k}\right)  =L\left(   N\cap I_{\Sigma} \right)
\end{equation}
it follows that the sequence $\left\{\Sigma^k\right\}$ satisfies property (a2) of Proposition \ref{basicLemma}.

In order to check property (a1) of Proposition \ref{basicLemma}, that is, $\Sigma^{k}$ is a $CAT(0)$ space, we only need to check that the angle around each vertex $v$ of $P^{k}$
is $\geq2\pi.$ Let $v_{A}$ and $v_{B}$ be the vertices in $P_{A}^{k}$ and
$P_{B}^{k}$ respectively corresponding to $v,$ and $\widehat{v_{A}}$ and $\widehat{v_{B}}$ the corresponding angles in $\Sigma_{A}^{k}$ and $\Sigma_{B}^{k}$ resp. Then $2\pi-\widehat{v_{A}}$ is the angle subtended at $v_{A}$ in the complement $\left(\Sigma_{A}^{k}\right)^c$ of $\Sigma_{A}^{k}.$
By (\ref{epsilon_interval_2riem})
\[ -\overline{\kappa}_{A}\left(  v_A\right)>\overline{\kappa}_{B}\left(  v_A\right)   \]
so  Corollary \ref{102Cor}(2) implies, for $k$ sufficiently large, that  $2\pi-\widehat{v_{A}}<\widehat{v_{B}}.$
It follows that
 $\widehat{v_{A}}+\widehat{v_{B}}>2\pi $ which shows that $\Sigma^{k}$ for all sufficiently large $k$ is a $CAT(0)$ space.

 The proof that the sequence $\left\{\Sigma^k\right\}$ satisfies property (a3) of Proposition \ref{basicLemma} is identical with the one given Theorem \ref{main} and so is the proof in Case 2 where $w$ is a point where the curvatures satisfy
 \[ \left \vert \overline{\kappa}_{A}\left(  w\right)  \right \vert -\kappa_{B}\left(w \right) =0 .\]
\end{proof}
\noindent Conflict of Interest statement: On behalf of all authors, the
corresponding author states that there is no conflict of interest.\newline%
\noindent Data Availability Statement: Data sharing not applicable to this
article as no datasets were generated or analyzed during the current study.

\end{document}